\documentclass{amsart}

\usepackage{latexsym,enumerate}
\usepackage{amsmath,amsthm,amsopn,amstext,amscd,amsfonts,amssymb}
\usepackage[ansinew]{inputenc}
\usepackage{verbatim}

\usepackage{epsfig}

\usepackage{graphicx}

\usepackage{epsfig,graphicx}

\newcommand{\CC}{\mathbb{C}}
\newcommand{\RR}{\mathbb{R}}
\newcommand{\ZZ}{\mathbb{Z}}
\newcommand{\DD}{\mathbb{D}}
\newcommand{\p}{\partial}
\newcommand{\F}{\mathcal{F}}
\renewcommand{\S}{\mathcal{S}}

\newtheorem{teo}{Theorem}[section]
\newtheorem{lema}[teo]{Lemma}
\newtheorem{prop}[teo]{Proposition}
\newtheorem{coro}[teo]{Corollary}
\newtheorem{definicion}[teo]{Definition}

\newtheorem{obs}[teo]{Remark}

\DeclareMathOperator{\Arccosh}{Arccosh}
\DeclareMathOperator{\Arctanh}{Arctanh}

\DeclareMathOperator{\inte}{int}
\DeclareMathOperator{\diam}{diam}

\DeclareMathOperator{\modulus}{mod}
\DeclareMathOperator{\card}{card}

\newcommand{\spb}[1]{\smallskip}
\newcommand{\mpb}[1]{\medskip}
\newcommand{\bpb}[1]{\bigskip}

\newcommand{\e}{\varepsilon}
\renewcommand{\d}{\delta}

\newcommand{\g}{\gamma}

\renewcommand{\l}{\lambda}

\renewcommand{\O}{\Omega}

\begin{document}

\title[THE TOPOLOGY OF BALLS AND GROMOV HYPERBOLICITY]
{THE TOPOLOGY OF BALLS AND GROMOV HYPERBOLICITY OF RIEMANN
SURFACES}

\author{
Jes\'us Gonzalo$^{(1)}$, Ana Portilla$^{(2)(3)}$, Jos\'e M. Rodr{\'\i}guez
$^{(3)}$ \and Eva Tour{\'\i}s$^{(2)(3)(4)}$}

\date{\today
\\
$(1)\,\,\,$ Partially supported by grants MTM2007-61982 from MEC
Spain and MTM2008-02686 from MICINN Spain.
\\
$(2)\,\,\,$ Supported in part by a grant MTM 2009-12740-C03-01 from
MICINN Spain.
\\
$(3)\,\,\,$ Supported in part by two grants MTM 2009-07800 and MTM
2008-02829-E from MICINN Spain.
\\
$(4)\,\,\,$ Supported in part by a grant CCG08-UC3M/ESP-4516 from
U.C.III$\,$M./C.A.M. Spain.}

\email{aportil2@slu.edu, jomaro@math.uc3m.es,
etouris@math.uc3m.es, jesus.gonzalo@uam.es}

\begin{abstract}
For each $k>0$ we find an explicit function $f_k$ such that the
topology of $S$ inside the ball $B_S(p,r)$ is `bounded' by $f_k(r)$
for every complete Riemannian surface (compact or noncompact) $S$
with $K \ge -k^2$, every $p \in S$ and every $r>0$. Using this
result, we obtain a characterization (simple to check in practical
cases) of the Gromov hyperbolicity of a Riemann surface $S^*$ (with
its own Poincar\'e metric) obtained by deleting from one original
surface $S$ any uniformly separated union of continua and isolated
points.
\end{abstract}

\maketitle{}

\section{Introduction.}

\mpb

This paper has two parts. In the first one we give a result
bounding the topological complexity of metric balls in terms of
the geometry. The bound we obtain is quite precise, and as an
application we show in the second part some useful criteria for
Gromov hyperbolicity of the Poincar\'e metric on a Riemann
surface. In particular (Theorem~\ref{t:finitegenus}), if $S$ is a
surface with finite genus and we remove from $S$ any `uniformly
separated' closed set $E$, then the Poincar\'e metric on the
deleted surface $S\setminus E$ is hyperbolic if and only if $S$
was hyperbolic with its own Poincar\'e metric.

Bounding the topology in terms of the geometry is a natural topic of
research; to mention a few examples see \cite{G2}, \cite{GP},
\cite{GPW}. More concretely, it is shown in \cite{G2} that the
fundamental group of a compact $n$-manifold $M$ with sectional
curvature verifying $K \ge -k^2$ can be generated with less than $C$
elements, where $C$ is a constant which just depends on $n$, $k$ and
the diameter of $M$. Theorem~\ref{t:balls} below is the noncompact
analogue for surfaces; it bounds the number of generators of the
fundamental group of a metric ball $B(p,r')$ by a constant times the
gap between the two sides of the classical comparison inequality:
\begin{equation}\label{ball-length}
\mbox{\rm length}\,\partial B(p,r)\;\leq\;\frac{2\pi}{k}\sinh (kr)\,
,
\end{equation} where $r$ is slightly larger than~$r'$. The result
is sharp: when the bound in Theorem~\ref{t:balls} is an
equality the metric ball is a topological disk and its curvature
is constant. An essential ingredient in the proof of this Theorem
is a second order differential inequality (\ref{fundamental})
relating the area and Euler characteristic of a metric ball, as
functions of the radius.

In Sections $5$ and $6$ we apply this result to the theory of Gromov
hyperbolicity. A geodesic metric space is called hyperbolic (in the
Gromov sense) if \emph{geodesic triangles are thin.\/} This means
that there exists an upper bound (the hyperbolicity constant) for
the distance from every point in a side of any geodesic triangle to
the union of the other two sides (see Definition \ref{def:Rips}).

Gromov hyperbolic spaces are a useful tool for understanding the
connections between graphs and Potential Theory on Riemannian
manifolds (see e.g. \cite{ARY}, \cite{CFPR}, \cite{FR2}, \cite{HS},
\cite{K1}, \cite{So}). Besides, the concept of Gromov hyperbolicity
grasps the essence of negatively curved spaces, and has been
successfully used in the theory of groups (see e.g. \cite{GH},
\cite{G3} and the references therein).

In recent years many researchers have been interested in studying
Gromov hyperbolicity of the metric spaces which appear in Geometric
Function Theory. In \cite{Be} and \cite{KN} it is shown that the
Klein-Hilbert metric is Gromov hyperbolic (under particular
conditions on the domain of definition); in \cite{Ha} it is proved
that the Gehring-Osgood metric is Gromov hyperbolic, and that the
Vuorinen metric is not Gromov hyperbolic except for a particular
case. In \cite{BB} significant progress is made about the
hyperbolicity of Euclidean bounded domains with their
quasihyperbolic metric (see also \cite{BHK} and the references
therein).

The study of Gromov hyperbolicity of a Riemann surface with its
Poincar\'e metric is non-trivial. An obvious reason is that
homological `holes' may be surrounded by geodesic triangles which
are not thin. For example in the `infinite grille', a
$\ZZ^2$-covering of the genus-2 surface, triangles engulfing many
holes are quite `fat'. An even stronger reason is the result, proved
in \cite{RT3}, that the usual classes: $O_G$, $O_{HP}$, $O_{HB}$,
$O_{HD}$, and surfaces with linear isoperimetric inequality, are
logically independent of the Gromov hyperbolic class. More
precisely, in each of these classes, as well as in its complement,
some surfaces are Gromov hyperbolic and some are not (even in the
case of plane domains). This has stimulated a good number of works
on the subject, e.g. \cite{APRT}, \cite{HLPRT}, \cite{HPRT1},
\cite{HPRT2}, \cite{PRT2}, \cite{PRT3}, \cite{PT}, \cite{RT3}.

A characterization of Gromov hyperbolicity for a surface $S^*$ with
cusps and/or funnels was obtained in \cite{PRT2} and \cite{PRT3}.
The idea there was to identify the cusps and funnels of $S^*$ with
pairwise disjoint compact sets $\{E_n\}$ removed from an original
surface $S$, so that the conformal structure of $S^*$ equals that of
$S\setminus \cup_n E_n$. Of course the Poincar\'e metric changes
when removing the sets $E_n$, but control of the resulting metric in
$S^*$ was achieved in terms of local information in $S$ near
each~$E_n$. Those two works use the idea of \emph{uniform
separation:\/} the $E_n$ are placed inside compact neighborhoods
$V_n\supset E_n$ having controlled topology each, and subject to
conditions such as $\p V_n$ being neither too long, nor too close to
$E_n$ or to the other $V_m$'s. The criterion obtained in \cite{PRT3}
additionally requires a `uniform hyperbolicity' condition, namely
that a single constant is valid for the hyperbolicity of all sets
$V_n\setminus E_n$ with the metric induced from $S^*$. This
condition can be hard to ensure in practice. In this paper we give
two criteria for hyperbolicity: Theorem~\ref{t:main}, based on
uniform separation but without the uniform hyperbolicity condition
in their hypotheses, and Theorem~\ref{t:infinite}, based on
`surrounding' the $E_n$'s by curves of controlled length (each $E_n$
is thus placed inside a ball, but with less constraints than in
uniform separation). In a nutshell, Theorem~\ref{t:main} states that
$S^*$ is hyperbolic if and only if $S$ is hyperbolic and a
reasonable metric condition on the handles of $S$ is held. There
follows an even simpler characterization when $S$ has either no
genus or finite genus, since in this case $S^*$ is hyperbolic if and
only if $S$ is hyperbolic (see Corollary \ref{c:main} and
Theorem~\ref{t:finitegenus}). This criterion has already received
important use in~\cite{HPRT2}.

Two ingredients have proved essential in the proofs of these
criteria. One is the above mentioned bound on the topology of balls.
Another ingredient consists on results about stability of
hyperbolicity; this means that we can set free some apparently
important quantities and still have uniform hyperbolicity. In this
direction, Theorem~\ref{t:finite} and Corollary~\ref{c:finite0}
guarantee uniform hyperbolicity even in situations where the
`punctures' $E_n$ approach one another. Likewise
Theorem~\ref{t:clasef} gives uniform hyperbolicity of all surfaces
with a fixed bound on the lengths of all of their funnel borders
\emph{except one} (see Definition~\ref{d:clasef}); this is a
remarkable improvement of a previous result \cite[Theorem
5.3]{PRT3}, where \emph{all} such borders had to be controlled.

It is also a remarkable fact that almost every constant appearing in
the results of this paper depends just on a small number of
parameters. This is a common place in the theory of hyperbolic
spaces (see e.g. \cite{GH}) and is also typical of surfaces with
curvature $-1$ (see e.g. the Collar Lemma in \cite{R} and \cite{S},
and Theorem~\ref{t:balls}). In fact, this simple dependence is a
crucial fact in the proof of Theorem~\ref{t:main}.

\mpb

\noindent {\bf Notations and conventions.} Every surface in this
paper is connected and orientable. In Section~2 we denote by $\p
B_r$ the \emph{extrinsic} boundary $\overline{B}_r\setminus B_r$ of
a ball $B_r$ as open set of an ambient surface $S$. In Sections 4,
5, and 6 we shall consider bordered $2$-dimensional manifolds, and
then the symbol $\p$, followed by the manifold's name, will indicate
the \emph{intrinsic} boundary of its bordered structure, e.g. $\p
(S^1\times (0,1])=S^1\times\{ 1\}$. If $(X,d_X)$ is a geodesic
metric space, we shall denote by $L_X$ the induced length, and given
$Y\subset X$ we shall write $d_X|_Y$, or simply $d_Y$, for the
geodesic distance induced on $Y$ by $d_X$. When there is no possible
confusion, we will not write the subindex $X$. Finally, we denote by
$c$, $k$, $c_j$, and $k_j$, positive constants which can assume
different values in different theorems.

\spb

\section{The topology of balls.}

\spb

In this Section we give upper bounds, in terms of the radius, for
the growth of the topological complexity of distance balls in a
surface endowed with a Riemann metric. Unlike the rest of the paper,
we allow the Gaussian curvature to be zero or positive somewhere.

Given a surface $S$, the topological complexity within $S$ of each
distance ball $B(p,r)$ will be measured using the integer $n(r)$
defined as follows:
\begin{equation}\label{n1n2}
n(r) := \;\hbox{\rm minimal number of generators for }\;\pi_1\big(\,
B(p,r)\, ,\, p\,\big)\; .
\end{equation}
Given $r_0$, we are going to bound $n(r')$ for some $r'>r_0$ not far
from $r_0$. This seems unavoidable because $n(r)$ is not always a
monotonic function of~$r$. Figure~\ref{figure:balls3} describes
metric balls $B_r$ such that $n(r)$ goes up and down as $r$ takes on
three values $r_1<r_2<r_3$. The starting ball $B_{r_1}$ (leftmost in
the figure) is diffeomorphic with a disk but its frontier in $S$ is
a curve with three points of self-tangency; it has $n(r_1)=0$. The
ball $B_{r_2}$ has $n(r_2)=3$ but one of its boundary components (a
`spurious hole') bounds a triangular disk in $S$. When $r=r_3$ the
triangular hole has disappeared and then~$n(r_3)=2$. In general, we
need to go from $B_{r_0}$ to some larger $B_{r'}$ with fewer
spurious holes.

\begin{figure}[h]
\includegraphics[scale=0.5]{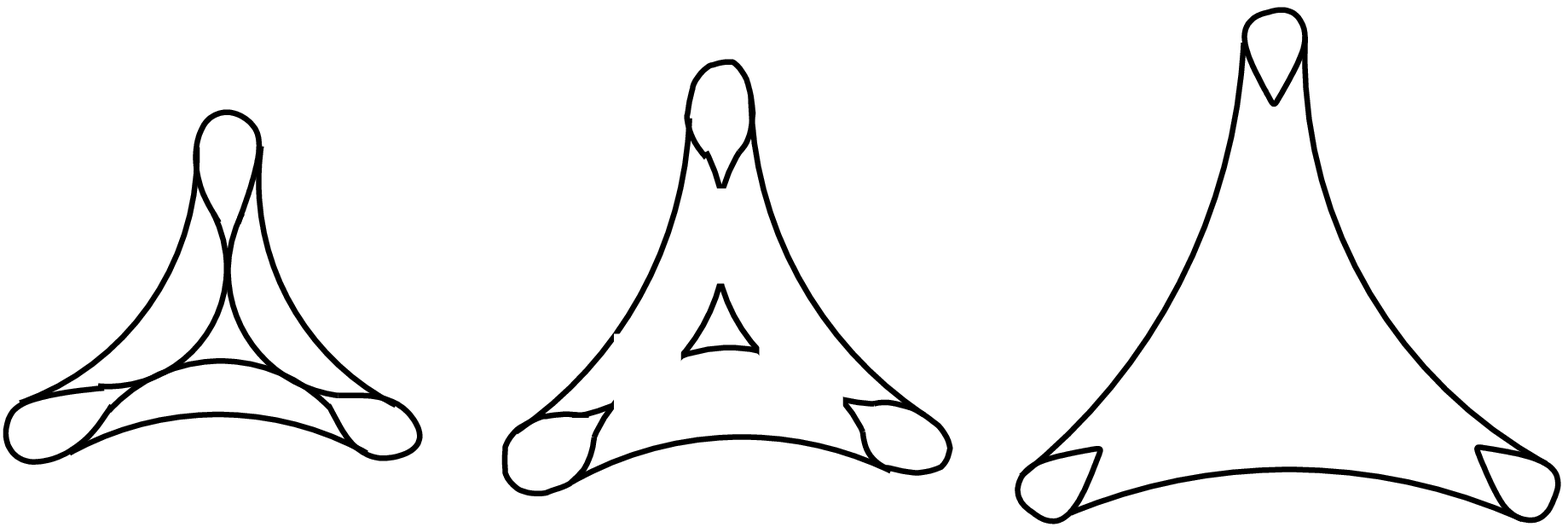}
 \caption{}
 \label{figure:balls3}
\end{figure}

\begin{obs}\label{con-borde}
{\rm In this section we always work within some closed ball
$\overline{B}(p,R)$ satisfying the following conditions:
\begin{list}{}{}
\item[--] The open ball $B(p,R)$ is not all of $S$. Thus for each $r<R$
the boundary $\partial B(p,r)$ has positive length.
\item[--] Every geodesic issuing from $p\,$ continues up to length
$R$. In particular, $\overline{B}(p,R)$ is compact.
\end{list}
}
\end{obs}

\begin{definicion}\label{ell}
For each $r\geq 0$ let $\ell (r)$ denote the length of the boundary
$\partial B(p,r)$.
\end{definicion}

\begin{teo}
\label{t:balls} Let $k,c$ be positive constants and assume $r_0>0$
is such that the ball $\overline{B}\big(\, p\, ,\,
r_0+\frac{c}{k}\,\big)$ is in the hypotheses of
Remark~\ref{con-borde}.

If the metric is real analytic and satisfies $K\ge -k^2$, or if it
is smooth and satisfies $0\ge K\ge -k^2$, then there is a radius
$r'$, strictly between $r_0$ and $r_0+\frac{c}{k}$, such that
\begin{equation}\label{top-estimate}
n(r')\;\leq\; \frac{1}{\sinh c}\left( \sinh
(kr_0+c)-\frac{k\,\ell\big(\, r_0+\frac{c}{k}\,\big)}{2\pi}\right)
\; .
\end{equation}
This inequality is also valid if $n(r')$ is defined using the
fundamental group of the closed metric $r'$-ball, and it is strict
unless the ball      $B\big(\, p\, ,\, r_0+\frac{c}{k}\,\big)$ is an
injective image of the exponential map (hence a disk) and has
$K\equiv -k^2$.
\end{teo}

In particular $n(r')\leq\displaystyle\frac{\sinh (kr_0+c)}{\sinh
c}<\frac{1}{1-e^{-2c}}e^{kr_0}$. Any general bound for the
fundamental group must grow exponentially with the radius: consider
copies of a fixed $Y$-piece with $K=-k^2$; if we paste them together
following the combinatorial design of a binary tree, we obtain an
example where $n(r)$ is asymptotically equal to $c_0e^{c_1r}$ for
some constants $c_0,c_1>0$.

Using Theorem~\ref{t:balls} one can improve the constant in
\cite[Theorem 3.1]{PRT2}, a result which says that Balls of small
radius (depending on the Gauss curvature bound) are simply or doubly
connected.

\begin{obs}\label{open-vs-closed}
{\rm For analytic metrics, and for those satisfying $K\leq 0$, we
are going to see that, as $r$ increases from $0$ to $R$, the
topology of the distance ball changes only at values of $r$ which
make up a discrete set in $[0,R)$. For all other values of $r$ the
inclusion $B(p,r)\hookrightarrow\overline{B}(p,r)$ is a homotopy
equivalence. The function $\,d(p\, ,\cdot )$ thus behaves like a
Morse function.}
\end{obs}

The present Section is organized as follows. We first examine in
depth the pertinent properties of the cut and conjugate loci. Then
we establish the regularity of the function $\ell (r)$ from
Definition~\ref{ell}, and give a formula for its derivative. After
these preliminaries we prove Theorem \ref{t:balls}, for which we
shall use a differential inequality (\ref{fundamental}) which
relates area and Euler characteristic of metric balls.

\mpb

Let $\mbox{\rm Exp}_p:T_pS\to S$ be the exponential map. The
boundary $\partial B(p,r)$ is some closed subset of the following
image
\[ \mbox{\rm Exp}_p\big(\, \{\; {\bf v}\in T_pS\; ;\; \|{\bf v}\|
=r\; \}\big)\, \] which is usually a complicated curve on $S$ with
many self-intersections. In particular, some parts of this image
will lie interior to the ball $B(p,r)$, not on its boundary.

\begin{definicion}
The {\em tangential cut locus} of $p$ is the set of vectors ${\bf
v}\in T_pS$ such that $\mbox{\rm Exp}_p(t{\bf v})$ defines a
minimizing segment for $t\in [0,1]$ and not for $t\in [0,T]$ if
$T>1$.  The {\em cut locus} of $p$ in $S$ is the image of the
tangential cut locus under $\mbox{\rm Exp}_p$, and its points are
called {\em cut points.}

The {\em tangential first conjugate locus} of $p$ is the set of
vectors ${\bf v}\in T_pS$ such that $\mbox{\rm Exp}_p$ has nonzero
jacobian at each $t{\bf v}$ with $t\in [0,1)$ and zero jacobian
at~$\bf v$. We then say that $\mbox{\rm Exp}_p({\bf v})$ is the
first conjugate point of $p$ along the geodesic with initial data
$p,{\bf v}$. The set of all such points, equal to the image of the
tangential first conjugate locus under $\mbox{\rm Exp}_p$, is
called {\em first conjugate locus} of $p$ in~$S$.
\end{definicion}

We work in a ball $\overline{B}(p,R)$ which is the exponential image
of the tangential ball $\overline{B}_R^{\mbox{\rm\scriptsize
T}}=\{\,{\bf v}\; ;\; \|{\bf v}\|\leq R\,\}$. Denote by $\mbox{\rm
Cut}^{\mbox{\rm\scriptsize T}}_p$ the part contained in
$\overline{B}_R^{\mbox{\rm\scriptsize T}}$ of the tangential cut
locus. Denote by $\mbox{\rm Cut}_p$ the part contained in
$\overline{B}(p,R)$ of the cut locus. Denote by $\mbox{\rm
Conj}^{\mbox{\rm\scriptsize T}}_p$ the part contained in
$\overline{B}_R^{\mbox{\rm\scriptsize T}}$ of the tangential first
conjugate locus. Denote by $\mbox{\rm Conj}_p$ the part contained in
$\overline{B}(p,R)$ of the first conjugate locus.

We base our discussion of these sets on the work of Myers
\cite{Myers1} and \cite{Myers2}, the reader may also see
\cite{Kobayashi} and \cite{Petersen}. If non-empty, the tangential
loci are described inside $T_pS$ as polar graphs:
\[
\mbox{\rm Conj}^{\mbox{\rm\scriptsize T}}_p = \{\, \|{\bf v}\|
=R_1(\theta )\,\} \quad ,\quad \mbox{\rm Cut}^{\mbox{\rm\scriptsize
T}}_p = \{\, \|{\bf v}\| =R_2(\theta )\,\} \; ,
\]
where $R_1(\theta )$ is smooth and $R_2(\theta )$ is continuous. For
$i=1,2$ the domain of $R_i(\theta )$ is either the whole unit circle
in $T_pS$, in which case the polar graph is a closed curve, or a
finite union of closed arcs in the unit circle, in which case the
polar graph is a finite union of embedded arcs with all the
endpoints on the outer circle
$\partial\overline{B}^{\mbox{\rm\scriptsize T}}_R$. These polar
graphs are compact and so are their exponential images $\mbox{\rm
Cut}_p$ and $\mbox{\rm Conj}_p$.

\begin{lema}\label{finite}
If the metric is real analytic and $\overline{B}(p,r)\neq S$, then
$\mbox{\rm Cut}_p\cap \mbox{\rm Conj}_p\cap \overline{B}(p,r)$ is a
finite set.
\end{lema}

\begin{proof} It is proved in \cite[Lemma 10]{Myers1} that any cut
point for $p$ which is also a conjugate point must be the
exponential image of a vector which is a local minimum for the norm
$\|\cdot\|$ in $\mbox{\rm Conj}^{\mbox{\rm\scriptsize T}}_p$. We
claim that the norm has finitely many local minima in $\mbox{\rm
Conj}^{\mbox{\rm\scriptsize
T}}_p\cap\overline{B}_r^{\mbox{\rm\scriptsize T}}$. Assume the
contrary, i.e. that there are infinitely many local minima for the
norm in $\mbox{\rm Conj}^{\mbox{\rm\scriptsize T}}_p\cap
\overline{B}_r^{\mbox{\rm\scriptsize T}}$; then infinitely many of
them belong to a single connected component $\mathcal C$ of
$\mbox{\rm Conj}^{\mbox{\rm\scriptsize T}}_p\cap
\overline{B}_r^{\mbox{\rm\scriptsize T}}$ and thus accumulate to
some vector ${\bf v}_0\in{\mathcal C}$ with $\|{\bf v}_0\|\leq r$.
The real analytic curve ${\mathcal C}$ and the circle centered at
$\bf 0$ and passing through ${\bf v}_0$ have a contact of infinite
order at ${\bf v}_0$, hence they coincide. But then $\mbox{\rm
Exp}_p$ sends that circle to a single point and
$\overline{B}(p,\|{\bf v}_0\| )$ is all of $S$, thereby forcing
$\overline{B}(p,r)$ to also be all of~$S$.\end{proof}

The results in \cite{Myers1} and \cite{Myers2} describe the cut
locus of a point on a surface and how it is reached by minimizing
geodesic arcs starting at such point (this second part is what most
interests us here). Under our hypothesis (analytic metric or $K\leq
0$) the set $\mbox{\rm Cut}_p$ is an embedded graph in~$S$ with
finitely many vertices and finitely many edges in each ball
$\overline{B}(p,r)$ not equal to~$S$. The points on this graph can
be of three kinds:
\begin{itemize}
\item Vertices of multiplicity $1$, i.e. points at which only one
edge arrives. These are conjugate points for $p\,$, and so they do
not exist if $K\leq 0$ and they are finite in number if the metric
is analytic. Each of these vertices is joined to $p$ by only one
minimizing geodesic arc.
\item Points of multiplicity $2$. These are the points on the
interior of the edges.
\item Vertices of multiplicity $m\geq 3$, i.e. points at which
three or more edges arrive.
\end{itemize}

Each edge is an embedded arc in $S$, and it follows from \cite[Lemma
11]{Myers1} that it is smooth except perhaps at the conjugate points
that it may contain. Thus each edge is smooth except perhaps at
finitely many points. It is proved in \cite[page 97]{Myers2} that
every interior point of an edge, smooth or non-smooth, is joined to
$p$ by exactly two minimizing geodesic arcs (of course, having the
same length). The same argument proves that if infinitely many edges
arrived at some cut point $q$ then there would exist infinitely many
minimizing geodesic arcs, all of the same length $r$, joining $p$
to~$q$ (which must then be conjugate to~$p$). If $K\leq 0$ this does
not happen because there are no conjugate points. If the metric is
real analytic then $\mbox{\rm Exp}_p$ would be a real analytic map
taking an infinity of tangent vectors at $p\,$, all with norm equal
to $r$, to the single point $q$. This would imply
$S=\overline{B}(p,r)$. Therefore a ball $\overline{B}(p,r)$ not
equal to $S$ does not contain any vertex of infinite multiplicity.
Once vertices of multiplicity $1$ are finite in number and vertices
of infinite multiplicity do not exist, the total number of vertices
and edges is finite due to topological reasons.

Let $\gamma (t)$, $1\leq t\leq 1$, be a geodesic with $\gamma (0)=p$
and ${\bf v}=\gamma'(0)$ a vector which is a local minimum for $r$
in the tangential first conjugate locus, then $q=\gamma
(1)=\mbox{\rm Exp}_p({\bf v})$ is a first conjugate point of $p$
along $\gamma$. Figure~\ref{figure:costura} shows the behavior near
$q$ of the geodesics which start at $p$ with initial velocity close
to~$\bf v$.

\begin{figure}[h]
\includegraphics[scale=0.8]{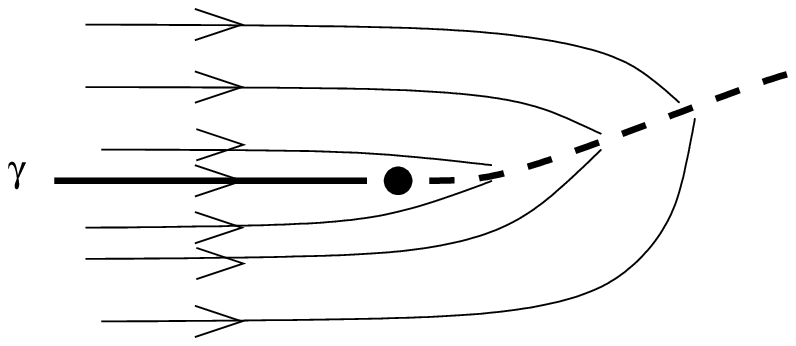}
 \caption{}
 \label{figure:costura}
\end{figure}

Consider a circular arc $C_0\subset T_pS$ centered at $\bf 0$ and
containing $\bf v$ as midpoint. Figure~\ref{figure:ortogonales}
shows the exponential image $\Gamma_q$ of $C_0$ as well as
orthogonal trajectories of the geodesics displayed in
Figure~\ref{figure:costura}, which trajectories are subsets of the
exponential images of circular arcs centered at $\bf 0$. If
$\gamma$ is the only minimizing path from $p$ to~$q$ then the
orthogonal trajectories shown in Figure~\ref{figure:ortogonales}
lie on the boundaries of balls centered at~$p\,$; in this case the
part of the cut locus on Figures~\ref{figure:costura}
and~\ref{figure:ortogonales} will be the dotted line. If there are
more minimizing paths from $p$ to~$q$ then some part of the
orthogonal trajectories will lie on the boundary and another part
will lie in the interior of the corresponding ball; in this case
the cut locus will have, in addition to the dotted line shown,
other branches ending at~$q$.

\begin{figure}[h]
\includegraphics[scale=0.5]{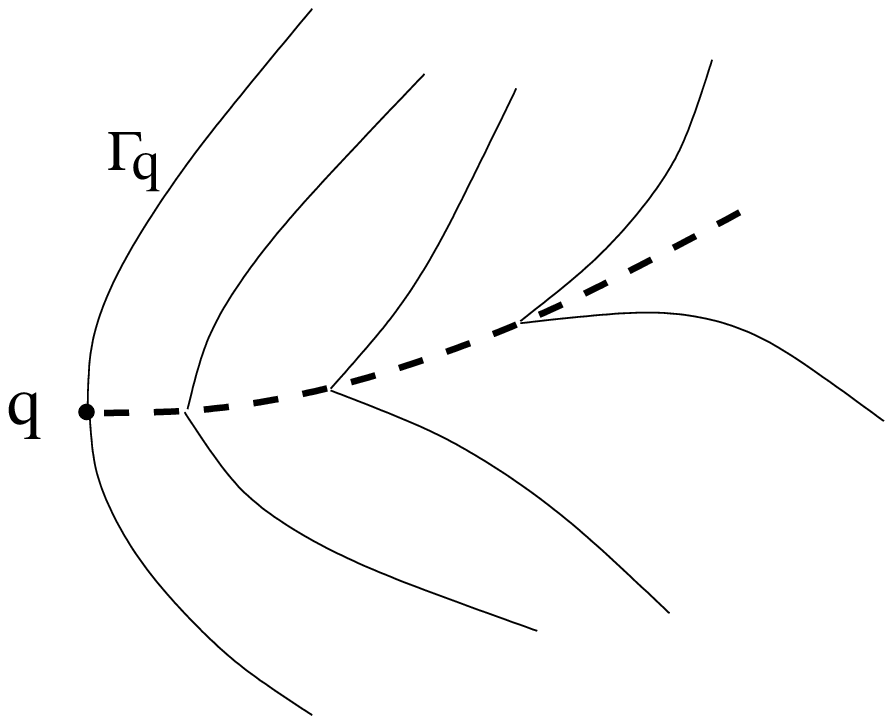}
 \caption{}
 \label{figure:ortogonales}
\end{figure}

\begin{obs}\label{inf-curv}
{\rm Since the geodesics in Figure~\ref{figure:costura} meet in
pairs making an angle which tends to zero as the cut point tends to
$q$, the image $\Gamma_q=\mbox{\rm Exp}_p(C_0)$ is of class
${\mathcal C}^1$ at the point $q$, but not of class~${\mathcal
C}^2$. In fact the geodesic curvature of $\Gamma_q$ at $q$ (defined
as limit of the curvatures at points close to $q$) is a positive
infinite multiple of~$\gamma'(1)$ because small arcs of $\Gamma_q$
around $q$ are supported by distance circles of arbitrarily small
radius centered at points of the dotted line.}
\end{obs}

We next define a type of point which is of great importance in our
context.

\begin{definicion}
A {\em middle point} is a cut point $q$ at which two minimizing
geodesic arcs (of equal length) issued from $p$ meet `head on', i.e.
the velocities of the two geodesic arcs at $q$ are each a negative
multiple of the other.
\end{definicion}

If $\gamma_1$ and $\gamma_2$ are those two minimizing arcs, then
$\gamma_1$ followed by reversed $\gamma_2$ defines a geodesic loop
based at~$p$ and having $q$ as middle point, hence the name.

\begin{lema}\label{coll-finite}
Suppose the metric is real analytic or it satisfies $K\leq 0$. If
the ball $\overline{B}(p,r)$ is not all of $S$, then there are only
finitely many middle points inside it.
\end{lema}

\begin{proof}
Suppose that, on the contrary, there is an infinity of such points.
Then there is an infinite sequence $\{\, ({\bf v}_n,L_n)\,\}$ where
${\bf v}_1,{\bf v}_2,{\bf v}_3,\dots $ are unit vectors in $T_pS$
and $L_1,L_2,L_3,\dots $ are lengths bounded by the number $r$, so
that the points $\mbox{\rm Exp}_p(L_n{\bf v}_n)$ are all middle
points of loops based at~$p$. Therefore we have $\mbox{\rm
Exp}_p(2L_n{\bf v}_n)=p\,$ for all~$n$. Since $\{ L_n\}$ is bounded
we extract a subsequence, again denoted $\{\, ({\bf v}_n,L_n)\,\}$,
that converges to some pair $({\bf v}_0,L)$ with $L\leq r$.

We prove first that this cannot happen for an analytic metric. There
is an $\varepsilon >0$ such that for every unit vector ${\bf v}$
close enough to ${\bf v}_0$ the geodesic segment
\[ \Gamma ({\bf v}):=\{\,\mbox{\rm Exp}_p(t{\bf v})\; ;\;
2L-\varepsilon\leq t\leq 2L+\varepsilon\,\} \] is small and very
close to $p$. Fix one such $\varepsilon >0$ and consider the
function:
\[ f({\bf v}):=d\,\big(\, \Gamma ({\bf v})\, ,\, p\,\big)\; . \] Once
$\varepsilon$ is fixed, for $\bf v$ close enough to ${\bf v}_0$ this
distance is attained at a point interior to the segment $\Gamma
({\bf v})$, hence f is analytic in a small enough neighborhood $C$
of ${\bf v}_0$ in the unit circle. At the same time $f$ vanishes on
an infinite sequence of points of $C$ converging towards ${\bf
v}_0$, which forces $f\equiv 0$. The circular arc $C$ thus
determines a $1$-parameter family of geodesic loops based at $p$,
which must all have the same length by Gauss' Lemma. It follows that
for all ${\bf v}\in C$ we have $\mbox{\rm Exp}_p(2L{\bf v})=p$ while
the entire loop of length $2L$ with initial data $p,{\bf v}$ is
contained in $\overline{B}(p,L)$. By analytic prolongation, we
obtain that the exponential image $B$ of the tangential disk $\{
\|{\bf v}\|\leq 2L\}$ is contained in
$\overline{B}(p,L)\subseteq\overline{B}(p,R)$ and that $\mbox{\rm
Exp}_p$ maps the tangential circle $\{ \|{\bf v}\| =2L\}$ to~$p\,$.
But then we would have $B=S=\overline{B}(p,L)$, and
$\overline{B}(p,r)$ would be all of $S$ because $r\ge L$.

We now do the proof for a metric with $K\leq 0$. Assuming the middle
points $\mbox{\rm Exp}_p(L_n{\bf v}_n)$ to be pairwise distinct, the
vectors $L_n{\bf v}_n$ are pairwise distinct and may be assumed to
be all different from their limit. Then the sequence of vectors
${\bf w}_n=2L_n{\bf v}_n$ has a subsequence $\{{\bf w}_{n_k}\}$
which converges {\em tangentially} to ${\bf w}=2L{\bf v}_0$. This
means that not only is $\bf w$ the limit of $\{{\bf w}_{n_k}\}$, but
the unit vectors $({\bf w}_{n_k}-{\bf w})/\| {\bf w}_{n_k}-{\bf
w}\|$ also have a limit ${\bf u}\in T_pS$. Then $\bf u$ is a unit
vector whose image under the differential of $\mbox{\rm Exp}_p$ at
$\bf w$ is zero, thus causing $p=\mbox{\rm Exp}_p({\bf w})$ to be
conjugate to itself which is impossible if~$K\leq 0$.
\end{proof}

We want to describe the geometry of the boundaries of metric balls
centered at $p$, for a metric which is analytic or satisfies $K\le
0$. Any $q\in\partial B(p,r)$ which is not a cut point is joined
to $p$ by a unique geodesic arc of length~$r$, along which $q$ is
not conjugate to~$p$, and the boundary is smooth (analytic)
near~$q$; also $q$ is not a self-intersection point of the image
$\mbox{\rm Exp}_p\big(\{\,{\bf v}\in T_pS\, ;\, \|{\bf
v}\|=r\}\big)$. Thus the only special points the boundary can have
are the points it shares with the cut locus. We see in
Figure~\ref{figure:ortogonales} that the boundary develops a
corner when it hits an endpoint of the cut locus graph, but its
topology does not change. For a short while after that moment, the
corner angle varies but otherwise the geometry of the boundary
remains unchanged.

We are going to see that, as $r$ increases, the geometry of the
boundary $\partial B(p,r)$ changes only when said boundary hits a
cut point which is either a conjugate point, a middle point of
multiplicity~$2$, or a vertex of higher multiplicity in the cut
locus graph. For the moment let us see what happens when the
boundary hits a middle point of multiplicity~$2$. Let
$\gamma_1,\gamma_2$ be the two minimizing geodesic arcs with
$\gamma_1(0)=\gamma_2(0)=p$, $\gamma_1 (1)=\gamma_2(1)=q$, and
$\gamma_1'(1)=-\gamma_2'(1)$. For $i=1,2$ let $C_i$ be a small
circular arc centered at $\bf 0$ in $T_pS$ and having $\gamma_i'(0)$
as midpoint. The exponential images of $C_1,C_2$ are ${\mathcal
C}^1$ curves $\Gamma_1,\Gamma_2$ meeting tangentially at~$q$.

If $\Gamma_1$ and $\Gamma_2$ curve away from each other toward
$\gamma_2$ and $\gamma_1$ respectively, then as $r$ increases the
boundaries $\partial B(p,r)$ evolve near $q$ as in
Figure~\ref{figure:ojales}.

\begin{figure}[h]
\includegraphics[scale=0.3]{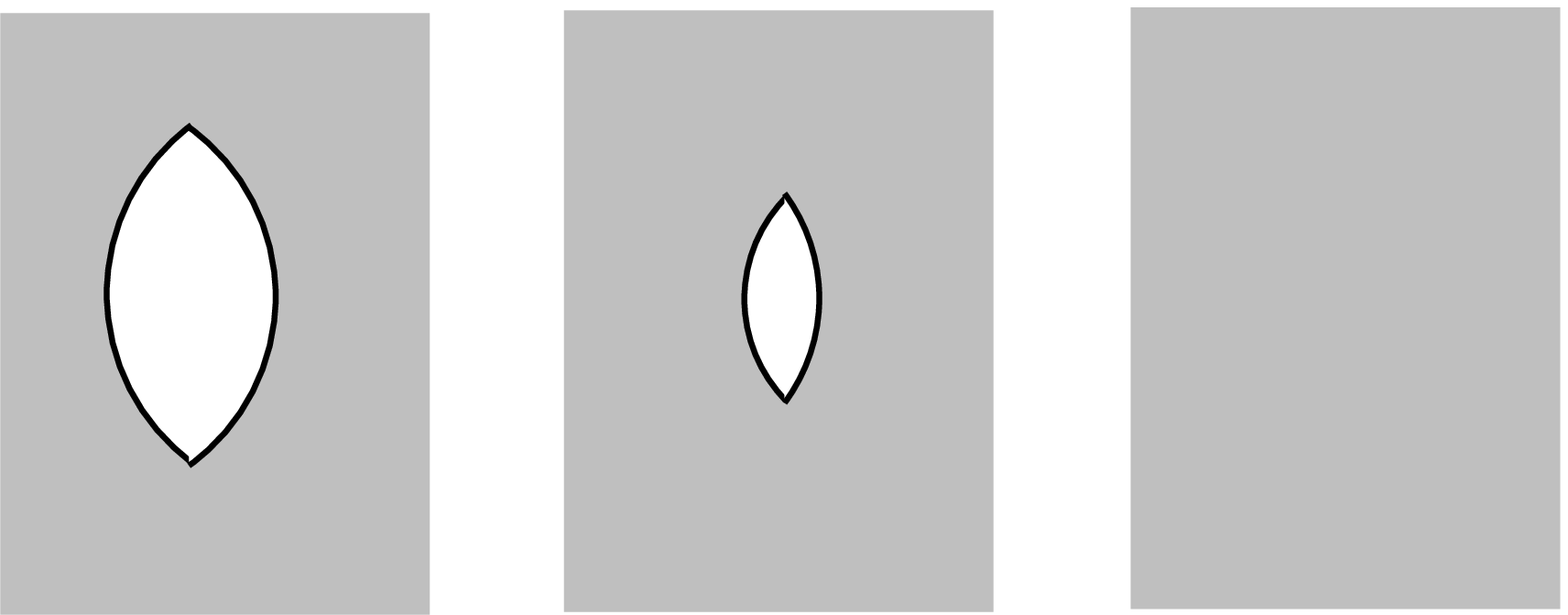}
 \caption{}
 \label{figure:ojales}
\end{figure}

If $\Gamma_1$ and $\Gamma_2$ curve away from each other toward
$\gamma_1$ and $\gamma_2$ respectively, then as $r$ increases the
boundaries $\partial B(p,r)$ evolve near $q$ as in
Figure~\ref{figure:beso}.

\begin{figure}[h]
\includegraphics[scale=0.3]{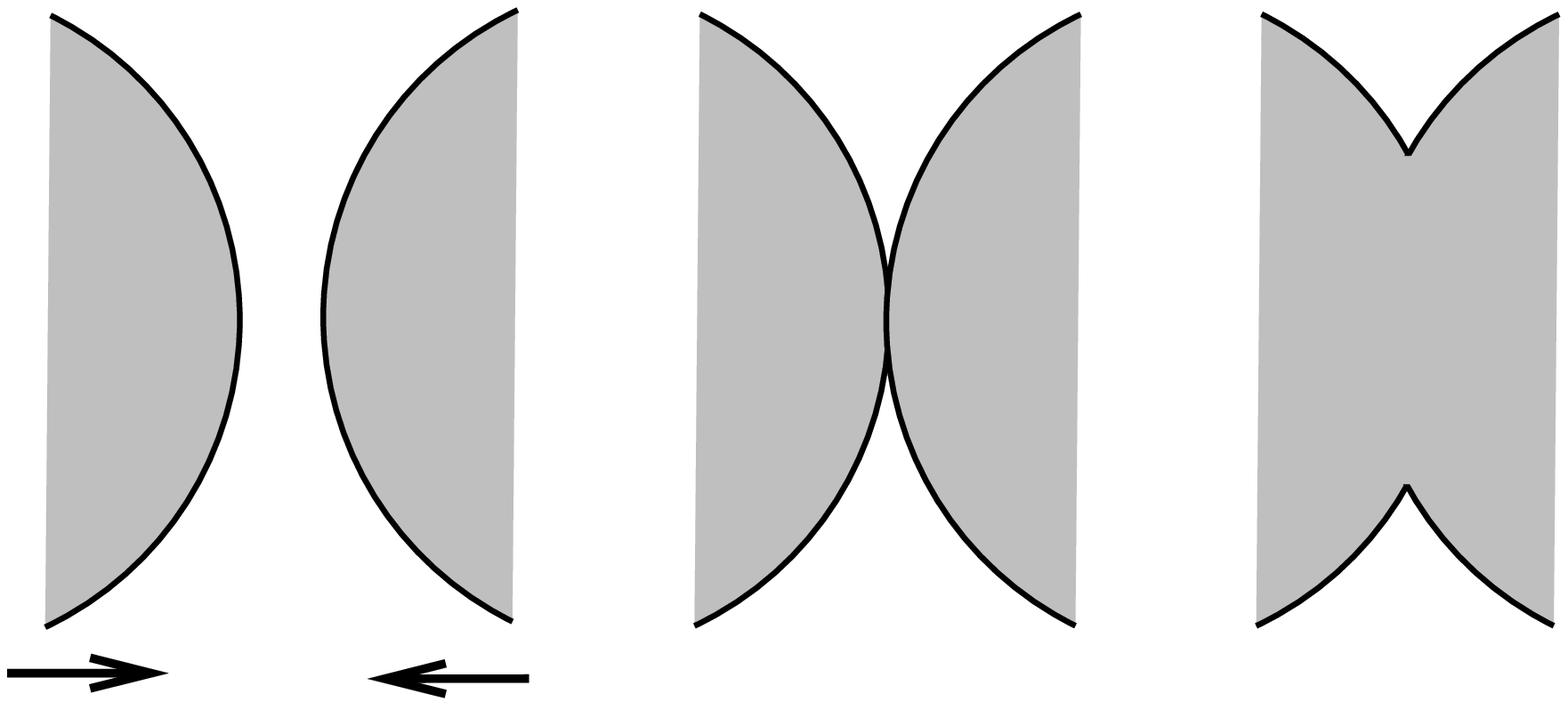}
 \caption{}
 \label{figure:beso}
\end{figure}

\begin{prop}\label{solo-beso}
If $K\leq 0$, then the only possible local geometry when the ball
boundary hits a middle point of multiplicity~$2$ is the one shown in
Figure~\ref{figure:beso}.
\end{prop}

\begin{proof}
In this case there are no conjugate points. The exponential map at
$p\,$ is a local diffeomorphism $T_pS\to S$ which pulls the metric
on $S$ back to a metric $\overline{g}$ on the tangent space. If
$(r,\theta )$ are orthonormal polar coordinates in
$T_pS\setminus\{{\bf 0}\}$, then $\overline{g}=dr^2+\lambda(r,\theta
)^2\, d\theta^2$ where $\lambda$ is a positive smooth function with
$\lim_{r\to 0}\lambda =0$, $\lim_{r\to 0}\lambda_r =1$, and
$-\lambda_{rr}/\lambda$ equal to the Gaussian curvature
of~$\overline{g}$. This yields $\lambda_{rr}\ge 0$ and $\lambda_r\ge
1$. The unit tangent vector to any Euclidean circle in $T_pS$
centered at $\bf 0$ is ${\bf t}=(1/\lambda)\partial_{\theta}$ and
one easily computes $\overline{g}\big(\,\nabla_{\bf t}{\bf t}\, ,\,
\partial_r\,\big) =-\lambda_r/\lambda <0$, hence said circle
curves strictly inward with respect to $\overline{g}$. Therefore the
boundary of any ball centered at $p\,$ in $S$ also curves strictly
inward at every non-corner point, and when it hits a middle point of
multiplicity~$2$ the only possible local geometry is the one shown
in Figure~\ref{figure:beso}, with the two colliding fronts having
finite non-zero curvature.
\end{proof}

The situation in Figure~\ref{figure:ojales} occurs in particular
when $q$ is conjugate to $p$ along both $\gamma_1$ and $\gamma_2$.
If $q$ is conjugate to $p$ along only one of these arcs, say
$\gamma_1$ to fix ideas, then $\Gamma_1$ curves towards $\gamma_2$
with infinite curvature at~$q$, while $\Gamma_2$ has finite
curvature at~$q$. The topology is then as in
Figure~\ref{figure:ojales} but there are several possibilities for
the geometry. One possibility is Figure~\ref{figure:ojales}. Another
possibility is the first image in Figure~\ref{figure:varios} (where
$\Gamma_2$ curves toward $\gamma_2$) with the shrinking hole now in
the shape of a crescent moon. Other possibilities (not depicted)
correspond to $\Gamma_2$ having zero curvature at~$q$.

\begin{figure}[h]
\includegraphics[scale=0.3]{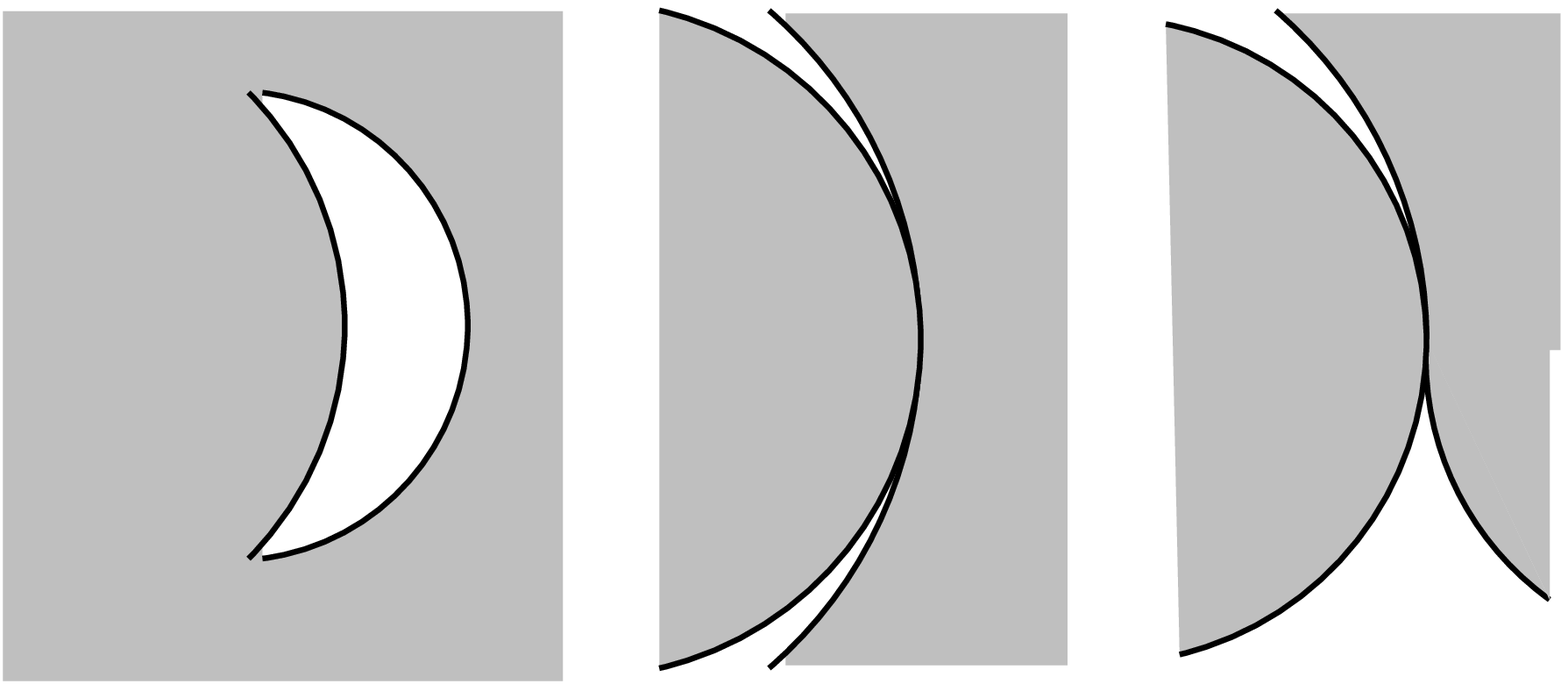}
 \caption{}
 \label{figure:varios}
\end{figure}

If $q$ is neither conjugate to $p$ along $\gamma_1$ nor along
$\gamma_2$, and the metric is real analytic, then
$\Gamma_1,\Gamma_2$ are real analytic arcs tangent at~$q$. If they
had a contact of infinite order at $q$, then they would coincide and
all of their points would be middle points. Thus
Lemma~\ref{coll-finite} implies that $\Gamma_1$ and $\Gamma_2$ have
a contact of finite order at~$q$. Therefore either $\Gamma_1$ stays
on one side of $\Gamma_2$, or these two arcs cross each other
(tangentially) at~$q$. If $\Gamma_1$ stays on one side of $\Gamma_2$
then the topology is as in Figures \ref{figure:ojales}
or~\ref{figure:beso}, but the geometry has more possibilities than
the ones shown in Figures \ref{figure:ojales} and~\ref{figure:beso}.
Some (not all) of the extra possibilities are shown in
Figure~\ref{figure:varios}. If $\Gamma_1$ crosses $\Gamma_2$
tangentially at~$q$ then the geometry is equal or very similar to
that in Figure~\ref{figure:uf}: the boundary containing $q$ has a
cusp at $q$, while the nearby boundaries contain a corner whose
branches make a nonzero angle.

\begin{figure}[h]
\includegraphics[scale=0.3]{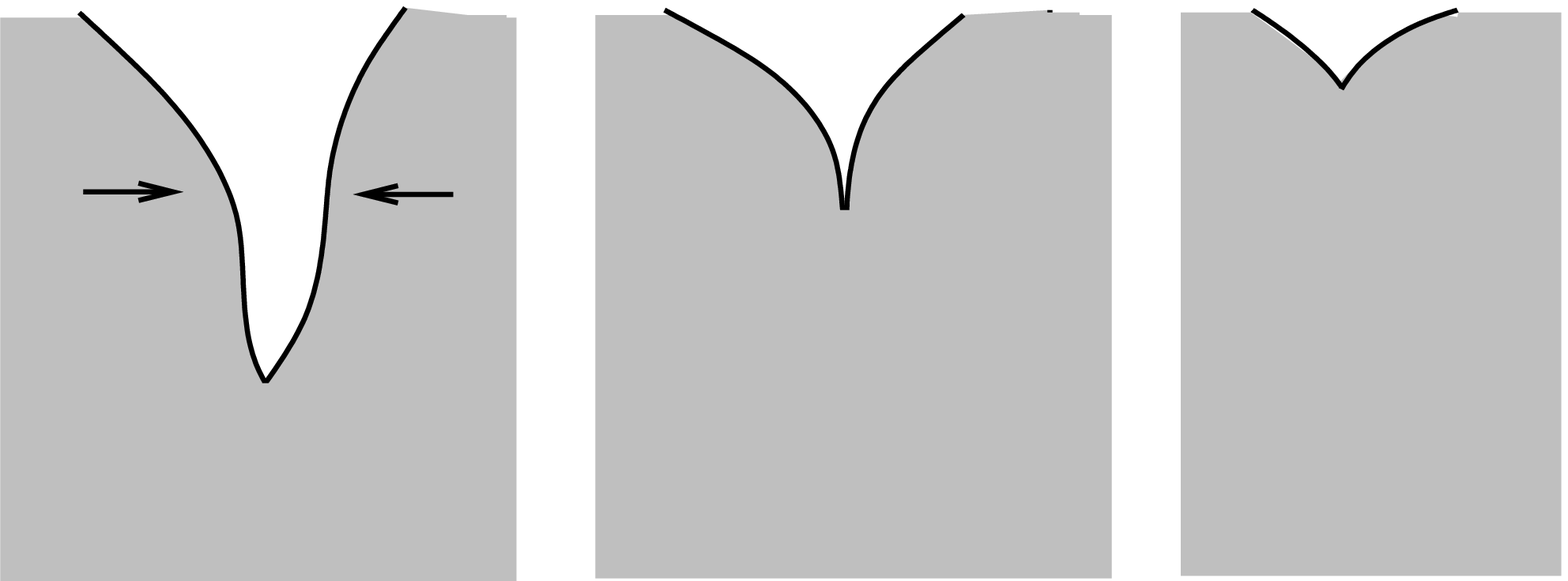}
 \caption{}
 \label{figure:uf}
\end{figure}

We briefly discuss now higher multiplicity vertices of the cut locus
graph. Myers shows in \cite{Myers1} and \cite{Myers2} that a vertex
of multiplicity $m$ is joined to $p$ by exactly $m$ minimizing
geodesic arcs. For example, Figure~\ref{figure:Y} describes how the
minimizing geodesic arcs issued from $p$ reach a $Y$-shaped part of
the cut locus; we see three geodesic arcs ending at the triple
point.

\begin{figure}[h]
\includegraphics[scale=0.6]{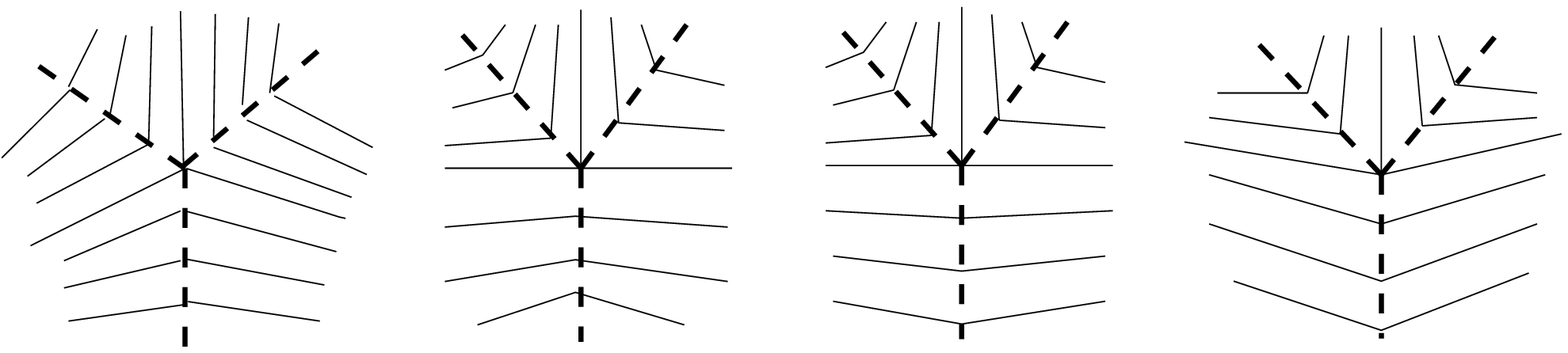}
 \caption{}
 \label{figure:Y}
\end{figure}

If two of the minimizing geodesics reaching the triple point make
flat angles when they meet (thereby making the triple point a middle
point as well) then the situations described in the second and third
images in Figure~\ref{figure:Y} are the only possible ones, because
middle points are isolated.

If the situation is as in the first or second image in
Figure~\ref{figure:Y} then the boundary evolves toward the triple
point as shown in Figure~\ref{figure:nablas}: we see a triangular
hole decreasing in size until it disappears.

\begin{figure}[h]
\includegraphics[scale=0.4]{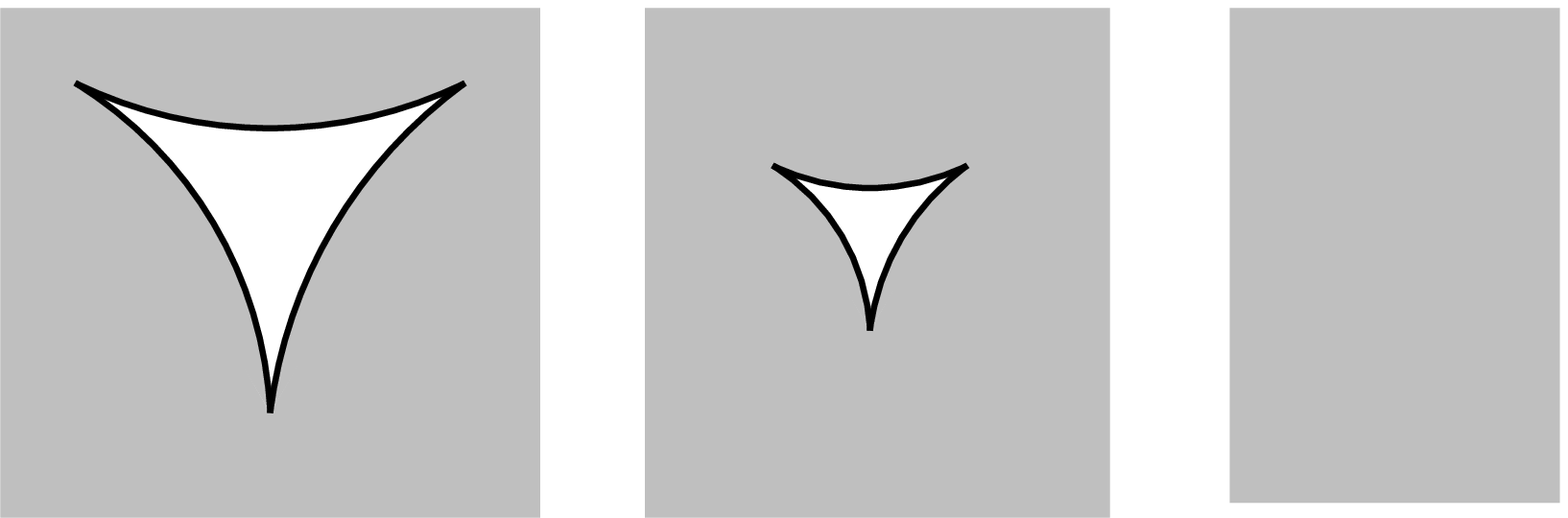}
 \caption{}
 \label{figure:nablas}
\end{figure}

In the case represented by the third image in Figure~\ref{figure:Y}
the boundary evolves as shown in Figure~\ref{figure:T}: it has
$3-1=2$ corners before hitting the triple point, one cusp when
hitting said point, and a single corner afterwards.

\begin{figure}[h]
\includegraphics[scale=0.3]{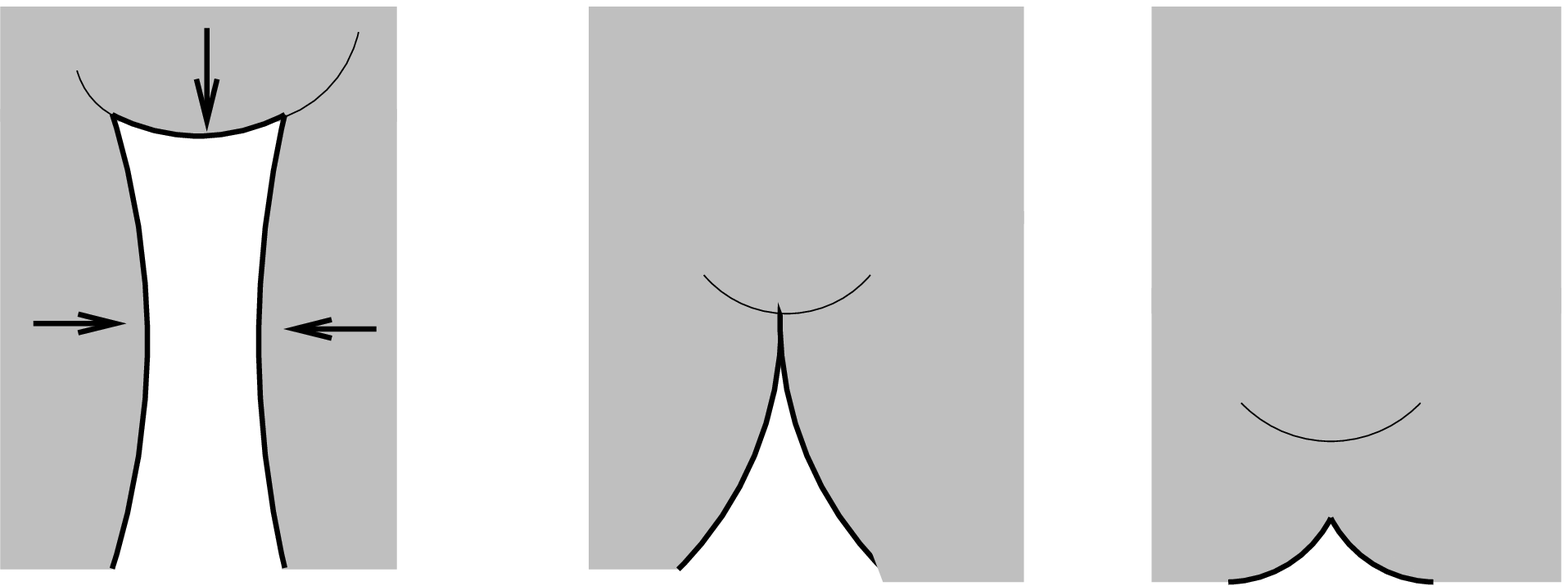}
 \caption{}
 \label{figure:T}
\end{figure}

The last image in Figure~\ref{figure:Y} occurs when two minimizing
geodesics make a concave angle at the triple point (measured without
going through the other geodesic). In this case the boundary evolves
almost like in Figure~\ref{figure:T}, the only difference being that
it also has a corner when hitting the triple point and so no cusp is
created in this case.

\begin{definicion}
The {\em interior angle} at a corner point of $\partial B(p,r)$ is
the angle between the two boundary branches ending at that point,
measured through the interior of~$B(p,r)$.
\end{definicion}

We now list together all possibilities when the boundary $\partial
B(p,r)$ touches an interior point $q$ of an edge of the cut locus,
i.e. a point of multiplicity~$2$. We then have exactly two
minimizing geodesic arcs joining $p$ to $q$, both of length $r$. In
particular, this is a self-intersection point of the image
$\mbox{\rm Exp}_p\big(\{\,{\bf v}\in T_pS\, ;\, \|{\bf
v}\|=r\}\big)$ and out of the four branches of this image that reach
$q$ only two branches are part of the boundary. We split this case
into two subcases:
\begin{itemize}
\item If $q$ is not a middle point then the part of $B(p,r)$
inside a small neighborhood of $q$ is the union $B'\cup B''$ of two
pieces whose boundaries near $q$ are at least ${\mathcal C}^1$
(recall the phenomenon in Figure~\ref{figure:ortogonales}) and
non-tangent at $q$. It follows that in such a case $\partial B(p,r)$
has a corner at $q$ and the interior angle at this corner is a
concave angle $\alpha\in (\pi ,2\pi )$. By Remark~\ref{inf-curv}, if
$q$ is also a conjugate point then at least one of the two branches
of the boundary has infinite curvature at~$q$. \item If $q$ is a
middle point then the boundaries of the two pieces $B',B''$ are
tangent at $q$, and one of the three phenomena described in Figures
\ref{figure:ojales}, \ref{figure:beso}, \ref{figure:varios},
\ref{figure:uf} occurs at~$q$. One phenomenon (see
Figure~\ref{figure:ojales} and the first image in
Figure~\ref{figure:varios}) consists on the boundary losing a small
connected component with two corners; the interior angles at the
corners remain inside $(\pi ,2\pi )$ during this process but both
tend to $2\pi$ as the hole's size tends to~$0$. Another phenomenon
(see Figure~\ref{figure:beso} and the last two images in
figure~\ref{figure:varios}) is an increase in the connectivity of
the ball; the interior angles at the created corners are both in
$(\pi ,2\pi)$ except at the instant when they are created. The third
possible phenomenon (see Figure~\ref{figure:uf}) is a cusp point on
the boundary when it touches $q$ and a single corner before and
after that instant; the topology remains unchanged during this
process.
\end{itemize}

Next we describe in detail what happens when the boundary
$\partial B(p,r)$ touches a point $q$ of multiplicity $m\geq 3$ in
the cut locus. This happens only for a finite number of values of
the radius~$r$ when the metric is analytic or with $K\leq0$,
because there is only a finite number of multiple points in such
cases. The minimizing geodesics from $p$ to $q$ make up a family
$\mathcal G$ with $m$ elements, all with the same length. Two
cases are possible:
\begin{itemize}
\item The angles at $q$ between consecutive geodesics in the
family $\mathcal G$ are all convex angles. Then $\partial B(p,r)$
has a $m$-sided polygonal component which shrinks down to the point
$q$ and then disappears. See Figure~\ref{figure:nablas} for the
$m=3$ case. (This multiple point will be a middle point if two {\em
non-consecutive} geodesics in $\mathcal G$ make flat angles).
\item There is a consecutive pair of geodesics in the family
$\mathcal G$ making a flat angle (in which case $q$ will be a middle
point) or a concave angle at~$q$. Then all other consecutive pairs
must make convex angles at~$q$. In this case the boundary $\partial
B(p,r)$ either has an $m$-sided polygonal component shrinking to~$q$
(Figure~\ref{figure:nablas} shows this for $m=3$) or its topology
remains unchanged during the process: $m-1$ corners before hitting
$q$,  one cusp or one corner at $q$ when hitting it, and a single
corner after hitting~$q$ (Figure~\ref{figure:T} shows this for
$m=3$).
\end{itemize}

\begin{obs}\label{corners}
A careful examination of the above study shows that the boundary
$\partial B(p,r)$ has corners, in each connected component of
positive length, forever after hitting the cut locus of~$p\,$.

The right-hand side of (\ref{ball-length}) is the obvious estimate
for the length of the exponential image of the $r$-circle on $T_pS$.
If $\partial B(p,r)$ has corners then said exponential image has
parts lying interior to $B(p,r)$. If the inequality $K\geq -k^2$ is
strict somewhere in $B(p,r)$ then the exponential image will have
length smaller that the right-hand side of~(\ref{ball-length}).
Therefore (\ref{ball-length}) can be an equality only if $B(p,r)$ is
disjoint from the cut locus of $p\,$ and $K\equiv -k^2$ inside
$B(p,r)$.
\end{obs}

Let $R$ be the radius in Remark~\ref{con-borde}. In $B(p,R)$ we
consider the set $\mathcal N$ which comprises all conjugate points
in the cut locus, all middle points of multiplicity~$2$ in the cut
locus, and all vertices of multiplicity $3$ or greater of the cut
locus. By the above discussion, if the metric is analytic or
satisfies $K\leq 0$ then the set $\mathcal N$ is finite inside each
ball $\overline{B}(p,r)$ with $r<R$, and as $r$ increases the
boundary $\partial B(p,r)$ gains (or loses) corners only by touching
this set. This implies finiteness of the number of corner points on
each ball boundary. Since the set $\mathcal N$ is finite in each
$\overline{B}(p,r)$ with $r<R$, the distances from the points in
$\mathcal N$ to $p\,$ can be arranged into an increasing sequence:
$r_1<r_2<r_3<\cdots$ which either is finite or converges to $R$ (the
latter can only occur if $\overline{B}(p,R)=S$). If $r<R$ is
different from these values then $\partial B(p,r)$ is a finite
disjoint union of simple closed curves, each having the corner
geometry just described, and the interior angle $\alpha$ at each
corner lies in the open interval $(\pi ,2\pi )$ and can thus be
written as $\alpha =\pi+2\beta$ for some $\beta\in
(0,\frac{\pi}{2})$. Also, for such $r$ the inclusion
$B(p,r)\hookrightarrow\overline{B}(p,r)$ is a homotopy equivalence
as claimed in Remark~\ref{open-vs-closed}.

\begin{lema}\label{finitos}
If the metric is real analytic or satisfies $K\leq 0$, then the
function $\ell (r)$ from Definition~\ref{ell} is continuous for all
$r\in [0,R)$ and smooth at $r\in [0,R)\setminus\{\,
r_1,r_2,r_3\dots\,\}$.
\end{lema}

\begin{proof}
In the case of a finite sequence $r_1,\dots ,r_s$, let $I$ be one of
the following intervals
\[ (0,r_1)\; ,\; (r_1,r_2)\; ,\; \cdots \; ,\; (r_{s-1},r_s)\; ,\;
(r_s,R)\; .\] In the case of an infinite sequence, let $I$ be
$(0,r_1)$ or any interval $(r_j,r_{j+1})$. In either case the number
of connected components and the corner geometry of the boundary
$\partial B(p,r)$ do not change while $r$ ranges over~$I$. Let $N_I$
be the number of maximal smooth segments in the boundary for $r\in
I$. These segments are the exponential images of $r\cdot
C_1(r),\dots ,r\cdot C_{N_I}(r)$ where $C_1(r),\dots ,C_{N_I}(r)$
are disjoint closed circular arcs in the tangential unit circle in
$T_pS$ (the rest of the tangential circle of radius $r$ is mapped by
$\mbox{\rm Exp}_p$ into the interior of the metric ball of
radius~$r$). The endpoints of the $C_i(r)$ are functions ${\bf
v}_j(r)$, $j=1,\dots ,2N_I$, with domain~$I$. For $r\in I$ the
number $\ell (r)$ is the integral of a smooth integrand over
$C_1(r)\cup\cdots\cup C_{N_I}(r)$, hence the smoothness of $\ell
(r)$ is equivalent to the smoothness of the endpoints ${\bf v}_j(r)$
of those circular arcs. Since $r\cdot C_1(r)\cup\cdots\cup r\cdot
C_{N_I}(r)$ is disjoint with the tangential first conjugate locus,
the functions ${\bf v}_j(r)$ are smooth if and only if the boundary
corner points $\mbox{\rm Exp}_p\big(r\cdot{\bf v}_j(r)\big)$ depend
smoothly on~$r$, which we next prove to be the case.

For each $C_i(r)$ let $C'_i(r)$ be an open circular arc containing
$C_i(r)$ such that $r\cdot C'_i(r)$ is still disjoint with the
tangential first conjugate locus. Then $b_i(r)=\mbox{\rm Exp}_p\big(
r\cdot C'_i(r)\big)$ is a smooth embedded arc in $S$ such that the
boundary corner points corresponding to $C_i(r)$ are the two
intersection points defined by $b_i(r)\cap\mbox{\rm Cut}_p\,$. Since
no boundary corner point is in $\mathcal N$, the cut locus is a
smooth embedded curve near them and so $b_i(r)\cap\mbox{\rm Cut}_p$
will depend smoothly on $r$ if $b_i(r)$ meets $\mbox{\rm Cut}_p$
transversally at these two corner points. The formula for first
variation of arc length implies that if $q$ is any smooth point of
the cut locus then said locus bisects the directions of the two
minimizing geodesic segments joining $p$ to $q$. The boundary
$\partial B(p,r)$ that goes through $q$ has two corner directions at
$q$ which are the orthogonal directions to those two minimizing
geodesic segments, hence the boundary has a direction tangent to the
cut locus at $q$ if and only if $q$ is a middle point (of
multiplicity~$2$). Since for $r\in I$ no corner point of $\partial
B(p,r)$ is a middle point, the smooth segment $b_i(r)$ meets the cut
locus transversally. This implies that the boundary corner points
are smooth functions of $r$ for $r\in I$ and, as explained above,
that $\ell (r)$ is smooth in~$I$.

The continuity of $\ell (r)$ at the special values $r_1,r_2,\dots$
follows by examination of Figures \ref{figure:ojales} to
\ref{figure:T} and the analysis that we made for each of them. For
example, in Figure~\ref{figure:ojales} we see a contribution to
$\ell (r)$ which has negative derivative for $r<r_k$, has derivative
equal to $-\infty$ at $r=r_k$, and equals $0$ for $r\geq r_k$.
Moreover this defines a H\"older continuous function of $r$ near
$r=r_k$ because we proved that the two fronts whose motion gives
rise to Figure~\ref{figure:ojales} either have different curvatures
at the special point (one of them infinite) or have only a finite
order contact at such point. Similar arguments apply to the other
Figures.
\end{proof}

We shall now give a formula for $\ell'(r)$. Let $k_g$ denote the
geodesic curvature of the boundary, taken with positive sign where
the boundary is curving towards the metric ball and with negative
sign where the boundary is curving away from the metric ball. The
integral $\int_{\partial B(p,r)}k_g\, ds$ is the contribution to the
derivative $\ell '(r)$ by the smooth segments of the boundary. To
determine the contribution from the corners it is sufficient to
consider the case of two straight segments lying on the Euclidean
plane and making an angle $\alpha =\pi+2\beta$. We see in
Figure~\ref{figure:a} that this corner contributes $\,
-2\tan\beta\,$ to $\ell'(r)$. The formula for the derivative of
$\ell (r)$ is:
\begin{equation}\label{ele-prima}
\ell'(r)\; =\; \sum_i -2\tan\beta_i +\int_{\partial B(p,r)}k_g\,
ds\quad ,\quad \mbox{\rm for }\; r\notin\{\,
0,r_1,r_2,r_3\dots\,\}\; ,
\end{equation}
where the index $i$ runs over the corners of the boundary
$\partial B(p,r)$ and $\alpha_i=\pi+2\beta_i$ are the respective
interior angles at those corners.

\begin{figure}[h]
\includegraphics[scale=0.5]{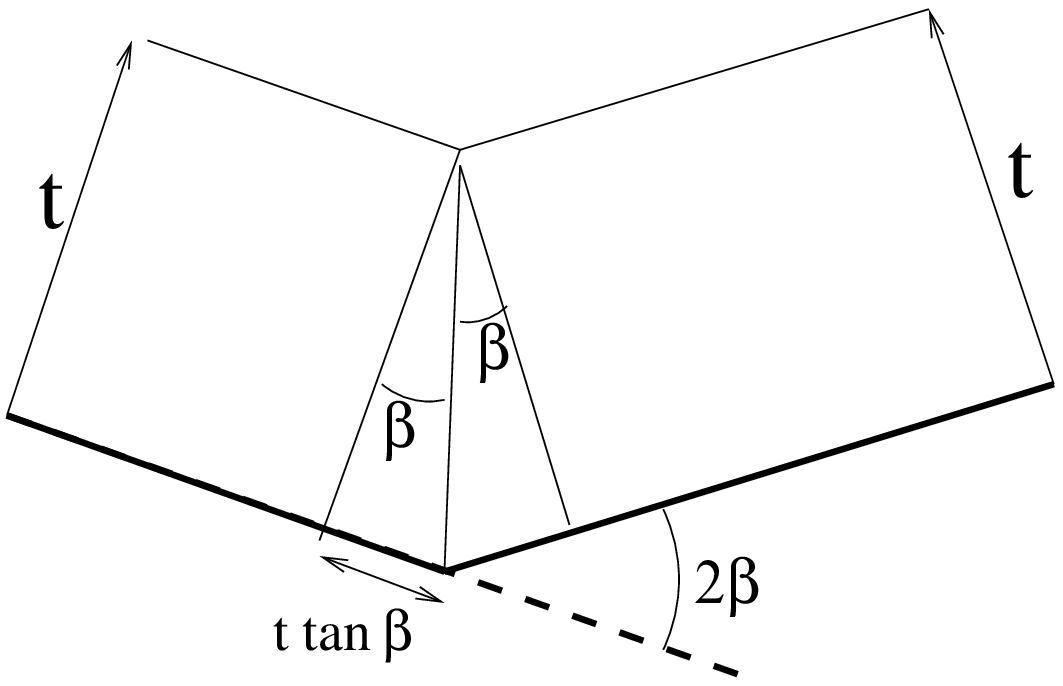}
 \caption{}
 \label{figure:a}
\end{figure}

\begin{definicion}\label{chi}
A surface is of \emph{finite type} if its fundamental group is
finitely generated.

Let $S$ be a connected surface of finite type, either non-compact or
compact with non-empty boundary; for such surfaces the
\emph{Euler-Poincar\'e characteristic} is the number $\chi
(S)=1-\mbox{\rm rank } H_1(S)$. Assume further that $\partial S$ is
either empty or a disjoint union of simple closed curves; then $\chi
(S)$ coincides with $2-2g-n$, where $g$ is the genus of $S$ and $n$
is the sum of the number of connected components of $\partial S$
plus the number of ends of $S$ that are homeomorphic with~$S^1\times
[0,\infty )$ (with the curve corresponding to $S^1\times\{ 0\}$
lying interior to~$S$).

\end{definicion}

\vspace{3mm}

\begin{proof}[Proof of Theorem~\ref{t:balls}.]
Define the integer-valued function $\chi (r)$ as follows: \[ \chi
(r)\; =\; \chi\big(\,  B(p,r)\,\big)\qquad \mbox{\rm for}\qquad 0<
r< r_0+\frac{1}{k}\; .\] We have three reasons for assuming that
$\overline{B}\big(\, p\, ,\, r_0+\frac{c}{k}\,\big)$ meets the
conditions of Remark~\ref{con-borde}. The first reason is that
formula~(\ref{ele-prima}) can then be used for
$0<r<r_0+\frac{c}{k}$. The second reason is that, once the balls
$B(p,r)$ have non-empty boundary for that range of values of $r$, we
can apply Definition~\ref{chi} to these balls. The third is to
ensure that the group $\pi_1\big(\, B(p,r)\, ,\, p\,\big)$ is free
in $n(r)$ generators; then $H_1\big(\, B(p,r)\,\big)$, being the
abelianization of $\pi_1\big(\, B(p,r)\, ,\, p\,\big)$, is isomorphic
with ${\mathbb Z}^{n(r)}$ and we have:
\begin{equation}\label{ji}
n(r)\; =\; \mbox{\rm rank } H_1\big(\, B(p,r)\,\big)\; =\; 1-\chi
(r)\, .
\end{equation}
In view of this, we seek an estimate for $1-\chi (r)$.

For $r\notin\{\, 0,r_1,r_2,r_3\dots\,\}$ we can use the Gauss-Bonnet
formula:
\[ 2\pi\,\chi (r) \; =\; \int_{B(p,r)}K\, d\,\hbox{area}
+\sum_i (\pi -\alpha_i)+\int_{\partial B(p,r)}k_g\, ds\, ,\] which
we rewrite as follows: \begin{equation}\label{Gau}  -\sum_i
2\beta_i+\int_{\partial B(p,r)} k_g\, ds + \int_{B(p,r)}K\,
d\,\hbox{area} \; =\; 2\pi\,\chi (r)\, .\end{equation} From
$\beta_i\in\big( 0,\frac{\pi}{2}\big)$ we infer
$\beta_i<\tan\beta_i$, which together with formulas
(\ref{ele-prima}) and (\ref{Gau}) leads to
\begin{equation}\label{ji-deriv}
\ell'(r)+\int_{B(p,r)}K\, d\,\hbox{area}\,\leq\; 2\pi\,\chi (r)
\quad ,\quad \mbox{\rm for }\; r\notin\{\, 0,r_1,r_2,r_3\dots\,\}\,
,
\end{equation} the inequality being strict whenever $\partial
B(p,r)$ has at least one corner. By Remark~\ref{corners}, this is
the case whenever $B(p,r)$ contains a cut point of~$p\,$.

Define now the function:
\[ a(r):\, =\, \mbox{\rm area}\,\big(\, B(p,r)\,\big)\; ,\]
which satisfies $a'(r)=\ell (r)$ for all $r\in \big[\, 0\, ,\,
r_0+\frac{c}{k}\,\big)$, thus:
\[ a(r)\in{\mathcal C}^1\big[\, 0\, ,\,
r_0+\mbox{$\frac{c}{k}$}\,\big)
\qquad\mbox{\rm and}\qquad a(r)\in{\mathcal C}^\infty\left(\,\big[\,
0\, ,\, r_0+\mbox{$\frac{c}{k}$}\,\big)\setminus\{\,
r_1,r_2,r_3\dots\,\}\,\right)\, .\]

Introduce now the hypothesis $K\geq -k^2$. The first consequence is
that formula~\ref{ji-deriv} yields the following differential
inequality:
\begin{equation}\label{fundamental}
a''(r)-k^2a(r)\; \leq 2\pi\,\chi (r) \quad ,\quad \mbox{\rm for }\;
r\notin\{\, r_1,r_2,r_3\dots\,\}\; ,
\end{equation}
which is strict whenever $B(r,p)$ contains a cut point of $p\,$. The
second consequence is the well-known bound (\ref{ball-length}) for
boundary length in terms of the corresponding length in a hyperbolic
plane with curvature~$-k^2$. The third consequence is the bound for
area:
\begin{equation}\label{ball-area}
a (r) \;\leq\; \frac{2\pi}{k^2}\,\big(\cosh (kr)-1\big)\,  ,
\end{equation}
obtained by integrating (\ref{ball-length}).

\begin{lema}\label{comparison}
Let $u(r),\overline{u}(r)$ be functions on an interval $r_0\leq r<
R$, smooth in the complement of a discrete set $Z\subset [r_0,R)$,
which satisfy:
\[ \left\{\begin{array}{ll}u(r) \in {\mathcal C}^1[r_0,R) & \\
u''(r)-k^2 u(r) = f(r) & r\notin Z\end{array}\right. \qquad\quad
 \left\{\begin{array}{ll}\overline{u}(r) \in {\mathcal C}^1[r_0,R) & \\
 \overline{u}''(r)-k^2 \overline{u}(r) = \overline{f}(r) & r\notin Z
 \end{array}\right. \] If the following inequalities hold
\begin{eqnarray}\label{ineq-1} f(r) &\leq & \overline{f}(r)
\;\;\mbox{\rm for all }\; r\notin F\, ,\\
\label{ineq-2} u(r_0) &\leq & \overline{u}(r_0)\, ,\\
\label{ineq-3} u'(r_0)-k\, u(r_0) &\leq &
\overline{u}'(r_0)-k\,\overline{u}(r_0)\, ,
\end{eqnarray}
then $\; u\leq\overline{u}\;$ and $\; u'-k
u\leq\overline{u}'-k\overline{u}\;$ (hence also
$u'\leq\overline{u}'$) everywhere on~$[r_0,R)$.

\end{lema}

\begin{proof}
Make the ansatz $u(r)\equiv e^{kr}c(r)$. Then the conditions imposed
on $u(r)$ are equivalent to the following:
\[ c(r)\in{\mathcal C}^1[r_0,R)\qquad\mbox{\rm and}\qquad
r\notin Z\Longrightarrow\frac{d}{dr}\big(\, e^{2kr}c'(r)\,\big)\;
=\; e^{kr}f(r) \, .\] The function $c'(r)\equiv e^{-kr}\big(\,
u'(r)-ku(r)\,\big)$ is continuous on $[r_0,R)$ and smooth for
$r\notin Z$. Since $Z$ is discrete, it follows that $c'(r)$ is given
at every $r\in [r_0,R)$ by the formula:
\[ c'(r)\; =\; e^{-2kr}\,\left( e^{2kr_0}c'(r_0)+\int_{r_0}^r
e^{kt}f(t)\, dt\right) .\] We have the analogous formula for the
derivative of the function $\overline{c}(r)$ given by
$\overline{u}(r)\equiv e^{kr}\,\overline{c}(r)$. Then, in view of
(\ref{ineq-1}), the inequality $c'(r)\leq\overline{c}'(r)$ holds on
all of $[r_0,R)$ if it holds at $r=r_0$, which is the case thanks
to~(\ref{ineq-3}). Now $c(r_0)\leq\overline{c}(r_0)$ is equivalent
to~(\ref{ineq-2}), and $c(r)\leq\overline{c}(r)$ follows by
integration.

The claimed inequalities follow by multiplying
$c(r)\leq\overline{c}(r)$ and $c'(r)\leq\overline{c}'(r)$
by~$e^{kr}$.
\end{proof}

\begin{obs}
The above proof gives $u'(r)<\overline{u}'(r)$ if we have
$f<\overline{f}$ in some non-trivial interval contained in
$[r_0,r]$.
\end{obs}

Consider the interval $I_0=\big[\, r_0\, ,\, r_0+\frac{c}{k}\,\big)$
and the constant:
\[ \chi_0:=\max_{r\in I_0}\chi (r)=1-\min_{r\in I_0}n(r)\, .\]
There is an $r'\in\big(\, r_0\, ,\, r_0+\frac{c}{k}\,\big)$ such
that $\chi_0$ equals the Euler characteristic of both $B(p,r')$ and
$\overline{B}(p,r')$. Inequality~(\ref{top-estimate}) is equivalent
to the inequality:
\begin{equation}\label{conclusion}
1-\chi_0\; <\: \frac{1}{\sinh c}\left( \sinh (kr_0+c)-
\frac{k\,\ell\big(\, r_0+\frac{c}{k}\,\big)}{2\pi}\right) .
\end{equation} Another property that the constant $\chi_0$ has is
that if we define the following two functions on $I_0$:
\begin{eqnarray*}
f(r) &=& a''(r)-k^2 a(r) \, ,\\ \overline{a}(r) &=& \left(\,
a(r_0)+\frac{2\pi\chi_0}{k^2}\,\right)\cosh\big(\, k(r-r_0)\,\big)
+\frac{1}{k}\,\ell (r_0)\sinh\big(\,
k(r-r_0)\,\big)-\frac{2\pi\chi_0}{k^2}\, ,
\end{eqnarray*}
then $\overline{a}$ is everywhere smooth, and by (\ref{fundamental})
it satisfies:
\[
\overline{a}''(r)-k^2\overline{a}(r)\;  =\; 2\pi\chi_0\;\geq\;
f(r)\qquad\mbox{for all }\; r\in I_0\setminus\{ r_1,r_2,\dots \}\, ,
\] the inequality being strict if $B(p,r)$ contains some cut point
of~$p\,$. We have adjusted $\overline{a}$ to satisfy
$\overline{a}(r_0)=a(r_0)\,$ and $\,\overline{a}'(r_0)=a'(r_0)$;
then Lemma~\ref{comparison} tells us that:
\begin{equation}\label{desig-1} \ell (r)\; =\; a'(r)\;\leq\;
\overline{a}'(r)\qquad \mbox{\rm for all }\; r\in I_0\,
,\end{equation} with strict inequality unless $B(p,r)$ contains no
cut point of~$p\,$.

Computing $\overline{a}'(r)$ from the explicit formula that defines
$\overline{a}$, and using (\ref{ball-length}) and (\ref{ball-area}),
one finds:
\begin{equation}\label{desig-2}
 \overline{a}'(r)\;\leq\;\frac{2\pi}{k}\left[\,\vphantom{\frac{a}{a}}\big(\,
\cosh (kr_0)-1+\chi_0\,\big)\sinh\big(\, k(r-r_0)\,\big)+\sinh
(kr_0)\cosh\big(\, k(r-r_0)\,\big)\, \right]\;\;\mbox{\rm for all
}r\in I_0\, .\end{equation}

Combining (\ref{desig-1}) and (\ref{desig-2}), we get for all $r\in
I_0$:
\[ 1-\chi_0 \; \leq\; \frac{1}{\sinh \big(\, k(r-r_0)\,\big)}\left(
\sinh (kr) -\frac{k\,\ell (r)}{2\pi}\right)  .\] Taking the limit as
$r\to r_0+\frac{c}{k}$, and using the following fact:
\[ \lim_{r\to
r_0+(c/k)}\ell (r)\;\geq\; \ell\big(\, r_0+\frac{c}{k}\,\big)\, ,
\]
we deduce (\ref{conclusion}).

Suppose (\ref{conclusion}) is an equality. Since $\chi_0=\chi\big(\,
B(p,r')\,\big)$, the ball $B(p,r')$ must then be disjoint from the
cut locus of $p\,$, hence diffeomorphic to a disk, and so $n(r')=0$.
But then (\ref{ball-length}) is an equality at $r= r_0+\frac{c}{k}$,
and by Remark~\ref{corners} the ball $B\big(\, p\, ,\,
r_0+\frac{c}{k}\,\big)$ is disjoint from the cut locus of $p\,$ and
has $K\equiv -k^2$.
\end{proof}

\section{Background on Gromov spaces.}

\spb

In our study of hyperbolic Gromov spaces we use the notations of
\cite{GH}. We give now the basic facts about these spaces. We refer
to \cite{GH} for more background and further results.

\spb

\begin{definicion}\label{hype}
Let us fix a point $w$ in a metric space $(X,d)$. Define the
\emph{Gromov product} of $x,y\in X$ with respect to the point $w$ as
$$
(x|y)_w:=\frac12\,\big( d(x,w)+d(y,w)-d(x,y) \big)\ge 0\,.
$$
We say that the metric space $(X,d)$ is $\d$-\emph{hyperbolic}
$(\d\ge 0)$ if
$$
(x|z)_w\ge\min\big\{ (x|y)_w, (y|z)_w \big\}-\d\,,
$$
for every $x,y,z,w\in X$. When we do not want to specify the value
of $\delta$, we say that $X$ is \emph{Gromov hyperbolic}.
\end{definicion}

\spb

It is convenient to remark that this definition of hyperbolicity is
not universally accepted, since sometimes the word `hyperbolic'
refers to negative curvature or to the existence of a Green
function. However, in this paper we only use the word {\it
hyperbolic} in the sense of Definition~\ref{hype}.

\spb

\noindent{\it Examples:}
\begin{list}{}{}
\item[(1)] Every bounded metric space $X$ is $(\diam X)$-hyperbolic.

\item[(2)] Every complete simply connected Riemannian manifold with
sectional curvature bounded from above by $-k^2$, with $k>0$, is
hyperbolic.

\item[(3)] Every tree with edges of arbitrary length is $0$-hyperbolic.
\end{list}
We refer the reader to \cite{BHK}, \cite{GH} and \cite{CDP} for
further examples.

\begin{definicion}
A metric space $X$ is a \emph{ geodesic metric space } if any two
points $x,y\in X$ can be joined by a path whose length equals
$d(x,y)$.
\end{definicion}

In general metric spaces, the \emph{length} $L(\g )$ of a path $\g
:[a,b]\to X$ is defined as $\,\sup\sum_{i=1}^n
d(\g(t_{i-1}),\g(t_{i})),$ taken over all partitions $\,
a=t_0<t_1<\cdots <t_n=b$.

\begin{definicion}
\label{def:Rips} In a general metric space $X$ a \emph{\/metric
geodesic\/} is a path $\g (t)$ such that
$d(\g(t),\g(s))=L(\g|_{[t,s]})=|t-s|$ for every $s,t\in [a,b]$, i.e.
$\g$ is minimizing and parametrized by arclength. We relax this
condition for a closed path: it only has to minimize length in its
free homotopy class.

If $T$ is a \emph{metric geodesic triangle} (i.e. its sides
$J_1,J_2,J_3,$ are metric geodesics) we say that $T$ is
$\d$-\emph{thin} if for every $x\in J_i$ we have that
$d(x,\cup_{j\neq i}J_{j})\le \d$. The space $X$ is $\d$-\emph{thin}
(or satisfies the \emph{Rips condition} with constant $\d$) if every
geodesic triangle in $X$ is $\d$-thin.
\end{definicion}

A basic result is that hyperbolicity is equivalent to the Rips
condition:

\begin{teo}
\emph{(\cite[p. 41]{GH})} Let us consider a geodesic metric space
$X.$

$(1)$ If $X$ is $\d$-hyperbolic, then it is $4\d$-thin.

$(2)$ If $X$ is $\d$-thin, then it is $4\d$-hyperbolic.
\end{teo}

From the next Section onwards, all spaces will be $2$-dimensional
Riemannian manifolds (with or without boundary) and length will be
defined as the obvious integral. Given such a surface $S$, its
distance function $d_S$ is defined by minimizing length of paths in
$S$. This turns $S$ into a geodesic metric space.

\begin{definicion}
\label{d:innermetric} For a sub-surface $X\subset S$ we have two
choices:
\begin{list}{}{}
\item[--] The \emph{extrinsic distance}, which is just $d_S$
acting only on pairs $(x,y)\in X\times X$. \item[--] The
\emph{intrinsic distance} $d_S|_X$, defined by minimizing
$d_S$-length of paths contained in~$X$. When there is no risk of
confusion we shall denote it~$d_X$.
\end{list} Likewise we have the \emph{extrinsic diameter}
$\,\diam_S(X)$ and the \emph{intrinsic diameter} $\,\diam_X(X)$.
\end{definicion}

Obviously $d_S\le d_X$ and $\,\diam_S(X)\le\diam_X(X)$. Notice also
that $X$ is always a geodesic metric space with the intrinsic
distance, not always with the extrinsic one.

\spb

Next we introduce a useful notion and use it to state
Theorem~\ref{t:treedecomp}, which will be important for the proof of
Theorem~\ref{t:clasef} below and was also used in the proofs of two
results, Theorems \ref{t:clases} and \ref{t:rsn}, which are quotes
from the previous work~\cite{PRT3}.

\begin{definicion}
\label{d:treedecomp} Let $(X,d)$ be a metric space, and let
$X=\cup_nX_n$ where $\{X_n\}_n$ is a family of connected geodesic
metric spaces such that $\eta_{nm}:=X_n\cap X_m$ are compact sets.
Further, assume that for any $n\neq m$ with
$\eta_{nm}\neq\varnothing$ the set $X\setminus\eta_{nm}$ is not
connected, and that the connected components of
$X\setminus\eta_{nm}$ containing $X_n\setminus\eta_{nm}$ are all
different from those containing $X_m\setminus\eta_{nm}$. We say that
$\{X_n\}_n$ is a $k$-\emph{tree decomposition} of $X$ if for each
$n$ we have $\sum_{m}\diam_{X_n} (\eta_{nm})\le k$.
\end{definicion}

If we define a graph with one vertex $v_n$ for each piece $X_n$, and
one edge $e_{nm}$ joining $v_n$ to $v_m$ if
$\eta_{nm}\neq\varnothing$, we obtain a tree. Hence the name.

\begin{teo}
\label{t:treedecomp} \emph{(Compare \cite[Theorem 2.9]{PRT1})} Let
us consider a metric space $X$ and a family of geodesic metric
spaces $\{X_n\}_n \subseteq X$ which is a $k$-tree
decomposition of $X$. Then $X$ is $\d$-hyperbolic if and only if
there exists a constant $c$ such that $X_n$ is $c$-hyperbolic for
every $n$. Furthermore, $\d$ (respectively $c$) is a universal
constant which only depends on $k$ and $c$ (respectively
$k$ and $\d$).
\end{teo}

\bpb

\section{Definitions and previous results on Riemann surfaces.}

\spb

In the following, Gaussian curvature is the constant $-1$. In this
Section we collect some definitions and facts concerning Riemann
surfaces which will be referred to afterwards.

\spb

An \emph{open non-exceptional} Riemann surface $S$ is the following
two things:
\begin{enumerate}
\item Conformally, it is a Riemann
surface whose universal covering space is the unit disk
$\DD=\{z\in\CC:\; |z|<1\}$. \item As a Riemannian manifold, it is
endowed with its own Poincar\'e metric, i.e. the metric obtained by
projecting the Poincar\'e metric of the unit disk $ds =2
|dz|/(1-|z|^2)$ down to $S$ by the covering map.
\end{enumerate}

\begin{obs}
{\rm (1) Note that, with this definition, every compact
non-exceptional Riemann surface without border is open.

(2) There are infinitely many metrics in the conformal class of $S$
with constant curvature $-1$, but the Poincar\'e metric is the only
complete one. In fact, a surface with a Riemann metric satisfying
$K\equiv -1$ has the unit disk as universal cover if and only if the
Riemann metric is complete.

(3) The only Riemann surfaces which are left out are the sphere, the
plane, the punctured plane and the tori. It is easy to study the
hyperbolicity of these particular cases.}
\end{obs}

\spb

A \emph{bordered non-exceptional Riemann surface} is a connected
$2$-dimensional Riemannian manifold  $S$ with boundary, subject to
the following restrictions:
\begin{enumerate}
\item It is the complement $S=R\setminus U$ of an open set $U$ in an
open non-exceptional Riemann surface $R$, and the Riemann metric on
$S$ is the one induced from the Poincar\'e metric of~$R$.
\item Mild border regularity conditions: the border of $S$ is locally
Lipschitz and any ball in $R$ intersects at most a finite number of
connected components of $U\!$.
\end{enumerate}
We say that $R$ contains $S$ isometrically. Since $S$ is a closed
subset of $R$, it is geodesically complete.

\begin{obs}{\rm If instead
of removing an open set we delete a closed set $E$ from an open
non-exceptional Riemann surface $R$, then we consider $R\setminus E$
also as an open non-exceptional Riemann surface, with its own
Poincar\'e metric which has also constant curvature $-1$ but is
longer than the (incomplete) Riemannian metric induced from~$R$.}
\end{obs}

Not every $2$-dimensional Riemannian manifold $S$ with $K\equiv -1$
embeds isometrically into an open non-exceptional Riemann surface.
For such isometric embedding to exist, the following necessary
condition must be satisfied: if $\widetilde{S}$ is the universal
cover of $S$, then any Riemann metric $g$ on $S$ with $K\equiv -1$
induces a local isometry $\Phi :\widetilde{S}\to{\mathbb D}$ that is
unique up to isometries of $\mathbb D$, and if $g$ is induced by an
embedding into an open non-exceptional Riemann surface then $\Phi$
must be injective. Let $S$ be an abstract closed disk, and take any
non-injective immersion $f:S\to{\mathbb D}$; then $f$ pulls the
Poincar\'e metric back to a metric $g$ on $S$ that has $K\equiv -1$;
but $(S,g)$ does not embed isometrically into any open
non-exceptional Riemann surface, because for this $g$ one has $\Phi
=f$. On the other hand $\,\inte S$ is an open non-exceptional
Riemann surface (with the conformal structure defined by $g$) and
the corresponding Poincar\'e metric is isometric with~$\mathbb D$,
quite different from the incomplete metric~$g|_{\inte S}$.

\begin{obs}
In this paper we only consider bordered non-exceptional Riemann
surfaces whose boundary components are (simple) closed curves.
\end{obs}

\begin{lema}\label{embedding}
If a doubly connected $2$-dimensional Riemannian manifold $S$ embeds
isometrically into some open non-exceptional Riemann surface, then
it embeds isometrically into either an annulus, or a cusp, or the
unit disk.
\end{lema}

\begin{proof}
Take any open non-exceptional Riemann surface $R$ containing $S$
isometrically. Consider the fundamental group $\pi_1(R)$ and the
subgroup ${\mathcal G}\subseteq\pi_1(R)$ defined by loops contained
in $S$. The group $\mathcal G$ is either trivial or infinite cyclic.
Consider also the covering map ${\mathbb D}\to R$.

If $\mathcal G$ is trivial then $S$ has a lift $S'\subset{\mathbb
D}$ which projects homeomorphically to $S$ under the covering map.
This $S'$ is isometric with $S$ and thus provides an isometric
embedding of $S$ into the unit disk.

If $\mathcal G$ is an infinite cyclic group $\langle\mu\rangle$,
then $S$ lifts to some region $S'\subset{\mathbb D}$ and the
restriction $S'\to S$ of the covering projection is equivalent to
the quotient map $S'\to S'/\langle\phi\rangle$, where $\phi$ is the
isometric action of $\mu\in\pi_1(R)$ on $\mathbb D$. The composition
$S\approx S'/\langle\phi\rangle\hookrightarrow{\mathbb
D}/\langle\phi\rangle$ is an isometric embedding.

The isometry $\phi$ cannot be elliptic, as it is induced by
$\mu\in\pi_1(R)$. If $\phi$ is a hyperbolic isometry, then ${\mathbb
D}/\langle\phi\rangle$ is an annulus. If $\phi$ is parabolic, then
${\mathbb D}/\langle\phi\rangle$ is a cusp.
\end{proof}

Lemma~\ref{embedding} enables us to classify, as Riemannian
manifolds, the bordered non-exceptional Riemann surfaces
homeomorphic with $S^1\times [0,\infty )$.

If $R$ is an annulus and $\g\subset R$ is an essential simple
closed curve, then the closure of each connected component of
$R\setminus\g$ is a doubly connected non-exceptional bordered
Riemann surface. Such bordered surfaces are called
\emph{generalized funnels,} as well as any bordered
non-exceptional Riemann surface isometric to any of these. A
\emph{funnel} is a generalized funnel whose boundary curve is a
geodesic.

If $R$ is a complete cusp and $\g\subset R$ is an essential simple
closed curve, then the closures $S_1,S_2$ of the connected
components of $R\setminus\g$ behave as follows:
\begin{enumerate}
\item On $S_1$ the curves freely homotopic to $\g$
become shorter as they run away from $\g$. We call $S_1$
\emph{narrow cusp end,} as well as any surface isometric with it. We
associate with $S_1$ an ideal point $q$ (the \emph{puncture}) such
that the conformal structure of $S_1$ extends to $S_1\cup\{ q\}$.
\item On $S_2$ the curves freely homotopic to $\g$ get
longer as they run away from $\g$. We call $S_2$, and any isometric
surface, \emph{wide cusp end.}
\end{enumerate}

A Jordan curve $\g\subset{\mathbb D}$ bounds two closed subsets in
the unit disk: a simply connected one and a doubly connected one.
We call the latter a \emph{disk end,} as well as any bordered
non-exceptional Riemann surface isometric to it.

Lemma~\ref{embedding} implies that if a bordered non-exceptional
Riemann surface $S$ is homeomorphic with $S^1\times [0,\infty )$
then it is isometric with one of the following: a generalized
funnel, a narrow cusp end, a wide cusp end, or a disk end (note that
$S$ is geodesically complete by definition).

\spb

Given a simple closed geodesic $\g$, we call the domain $\{p\in
S:\,d_S(p,\g)<d\}$ the \emph{collar of width $d$ about} $\g$ if it
is a tubular neighborhood of $\g$ with fibres the geodesics
orthogonal to $\g$. The Collar Lemma \cite{R}, \cite[Chapter 4]{Bu}
says that there exists a collar about $\g$ of width $d_0$, where
$\cosh d_0 = \coth (L_S(\gamma)/2)$ (the shorter the curve, the
thicker the collar). We shall use it in the proof of
Theorem~\ref{t:infinite}.

\begin{obs}\label{disjuntas}
Let $S$ be a non-exceptional Riemann surface, open or with compact
border. If $g_1$ is a closed curve neither freely homotopic to a
point, nor to a puncture, then the free homotopy class of $g_1$
contains a unique closed geodesic $\g_1$ which in fact minimizes
length in said class. If $g_1$ is simple then so is the geodesic
$\g_1$. If $g_1 ,g_2$ are simple, disjoint, and not freely
homotopic, then so are the corresponding geodesics $\g_1$
and~$\g_2$. In fact, if $L(g_1),L(g_2)\le l$ and $\cosh d_0=\coth
(l/2)$ then the collars of width $d_0$ about $\g_1$ and $\g_2$ are
disjoint (see \cite[Chapter 4]{Bu}).
\end{obs}

We next show that wide cusp ends and disk ends happen very rarely.

\begin{lema}\label{wide}
Let $S$ be a non-exceptional Riemann surface, open or bordered. If
$S$ has a wide cusp end, then it can only embed isometrically into a
cusp. If $S$ has a disk end, then it can only embed isometrically
into the unit disk.
\end{lema}

\begin{proof}
Suppose $R$ is an open non-exceptional Riemann surface containing
$S$ isometrically. Then $R$ shares with $S$ an end $S_0$ which is
either a wide cusp end or a disk end. The essential simple closed
curves in $S_0$ do not give rise to a non-constant geodesic in $R$,
because then $S_0$ would be part of a funnel end. By
Remark~\ref{disjuntas}, these curves must be freely homotopic in $R$
either to a puncture, which forces $R$ to be a cusp, or to a point,
which forces $R$ to be the unit disk.
\end{proof}

We call \emph{compact annulus} any non-exceptional bordered
Riemann surface homeomorphic with $S^1\times [0,1]$. Applying
Lemma~\ref{embedding} to these surfaces we obtain the following
result, which will be essential in the proof of
Theorem~\ref{t:clasef}.

\begin{prop}\label{RS}
Let $S$ be any bordered non-exceptional Riemann surface. There is a
canonical choice $R_S$ of an open non-exceptional Riemann surface
with the following properties: (1) $S$ embeds isometrically into
$R_S$, (2) $R_S$ has the same genus as $S$, and (3) $\chi
(R_S)\geq\chi (S)$.
\end{prop}

\begin{proof}
Let $\g$ be any connected component of $\partial S$ and take a
compact annulus $A\subset S$ bounded by $\g$ and another simple
closed curve~$\eta\subset S$. When one applies Lemma~\ref{embedding}
to $A$, there are five possibilities:
\begin{enumerate}
\item We find an isometric embedding $f:A\to{\mathbb D}$ and $f(\g
)$ lies in the interior of $f(\eta )$. In this case we can glue the
bounded component of ${\mathbb D}\setminus f(\eta )$ to $S$ by
overlapping $A\setminus\eta$ with itself, thereby producing a larger
Riemannian surface that has $S$ isometrically embedded inside it and
has one less boundary component than~$S$.
\item We find an isometric embedding $f:A\to{\mathbb D}$ and $f(\eta
)$ lies in the interior of $f(\g )$. Then we glue, with overlapping,
the doubly connected component of ${\mathbb D}\setminus f(\eta )$ to $S$
and produce a surface (open or bordered) which contains $S$
isometrically and has a disk end. By Lemma~\ref{wide} the latter
surface, and hence also $S$, embeds isometrically into the unit
disk. \item $A$ embeds isometrically into a cusp $R_0$, so that $\g$
is the boundary of the narrow end of $R_0\setminus\inte S$ and
$\eta$ is the boundary of the wide end. Then we can glue, with
overlapping, the narrow end of $R_0$ bounded by $\eta$ to $S$; this
produces a surface that contains $S$ isometrically and has the
boundary curve $\g$ replaced with a puncture. \item $A$ embeds
isometrically into a cusp $R_0$, so that $\g$ is the boundary of the
wide end of $R_0\setminus\inte S$ and $\eta$ is the boundary of the
narrow end. Then we glue, with overlapping, the wide end of $R_0$
bounded by $\eta$ to $S$; we thus obtain a surface that contains $S$
isometrically and has a wide cusp end. This and Lemma~\ref{wide}
imply that $S$ embeds isometrically into a cusp, so that $\g$ is an
essential curve in that cusp. This forces $\pi_1(S)$ to be
non-trivial. \item $A$ embeds isometrically into an annulus. Then we
find a generalized funnel which we can glue to $S$ with overlapping,
and thus produce a surface that contains $S$ isometrically and has
the boundary curve $\g$ replaced with a funnel end.
\end{enumerate}
If one boundary curve of $S$ is in case (2) we take $R_S={\mathbb
D}$. If one boundary curve of $S$ is in case (4) we make $R_S$ equal
to the cusp that contains $S$ isometrically; in this case $\chi
(R_S)=0\geq\chi (S)$ because $\pi_1(S)$ is non-trivial.

Suppose now that cases (2) and (4) do not occur for any connected
component of $\partial S$. The boundary $\p S$ consists of a
sequence $\g_1,\g_2,\dots$ of simple closed curves, and we can
choose the corresponding compact annuli $A_1,A_2,\dots$ pairwise
disjoint. Then we do, simultaneously for all $\g_i$, the gluing that
corresponds to each of them (described in the odd-numbered cases),
and we obtain a surface $R_S$ which is geodesically complete, with
$K\equiv -1$, with empty boundary, and with the same genus as $S$.
Moreover, $R_S$ contains $S$ isometrically inside it and satisfies
$\chi (R_S)\geq \chi (S)$, with strict inequality if case (1) has
occurred at least once. It is easy to see that $R_S$ does not depend
on the choice of the pairwise disjoint sequence $\{ A_i\}_i$.
\end{proof}

\begin{definicion}
Let us consider a non-exceptional Riemann surface $S$ of finite type
(open or with compact border) with $\chi (S)\leq 0$. An \emph{ outer
loop } in $S$ is either the boundary geodesic of a funnel or the
minimizing curve in the free homotopy class of some connected
component of~$\p S$. We consider punctures as outer loops of zero
length.
\end{definicion}

\begin{definicion}\label{d:clasef}
Fix a non-negative integer $a$ and a positive real number $l$. We
denote by $\F(a,l)$ the set of non-exceptional Riemann surfaces of
finite type $S$ verifying the following properties:

\begin{enumerate}
\item
S has no genus and $0\geq\chi (S)\geq -a$, equivalently $0\leq
n-2\leq a$, where $n$ is the number from Definition~\ref{chi}.
\item
If $\chi (S)=0$, then the unique outer loop has length less than or
equal to $l$. If $\chi (S)<0$ then every outer loop, except perhaps
one of them, has length less than or equal to $l$.
\item If $\p S$ is non-empty, then $L_S(\p S) \le l$.
\end{enumerate}

We denote by $\S(a,l)$ the set of Riemann surfaces $S\in \F(a,l)$
verifying that every outer loop has length less than or equal to
$l$. Notice that $\S (a,l)$ and $\F (a,l)$ coincide only for~$a=0$.
\end{definicion}

\begin{teo}
\emph{(\cite[Theorem 5.3]{PRT3})} \label{t:clases} For each $l\ge 0$
and each non-negative integer $a$, there exists a constant $\d$,
which only depends on $a$ and $l$, such that every surface in
$\S(a,l)$ is $\d$-hyperbolic.
\end{teo}

\begin{definicion}
An $N$-\emph{normal neighborhood} of a subset $F$ of a Riemann
surface $S$ is a compact, connected, bordered Riemann surface
without genus $V$ such that $F\subset V\subset S$, and $\p V$ is the
union of at most $N$ simple closed curves, i.e. $\chi (V)\geq 2-N$.

A set $E=\cup_n E_n$ in an open non-exceptional Riemann surface $S$,
with each $E_n$ compact, is called $(r,s,N)$-\emph{uniformly
separated} in $S$ if for every $n$ we can choose an $N$-normal
neighborhood $V_n$ of $E_n$ such that $V_n\setminus E_n$ is
connected, $d_S (\p V_n, E_n)\ge r$, and $L_S (\p V_n)\le s$, and
the whole sequence $\{ V_n\}_n$ can be chosen so that $d_S (V_n,
V_m)\ge r$ for every $n\neq m$.
\end{definicion}

\begin{obs}\label{separated}
As each $V_n$ has zero genus by definition, it is $V_n\in
\S(N-2,s)\subset \F(N-2,s)$ independently of~$r$. Also, Alexander
duality implies that if $E_n$ and $V_n\setminus E_n$ are connected,
then $E_n$ is simply connected.
\end{obs}

\spb

The uniformly separated sets play a central role in many topics in
Complex Analysis, such as interpolation in the unit disk $\DD$ (see
\cite{Ca}), harmonic measure (see \cite{OS}) and the study of linear
isoperimetric inequalities in open Riemann surfaces (see
\cite[Theorem 1]{APR} and \cite[Theorems 3 and 4]{FR1}).

\begin{definicion}\label{def:D}
Let $S$ be an open non-exceptional Riemann surface, $E=\cup_nE_n$ an
$(r,s,N)$-uniformly separated set in $S$ and $S^*:=S \setminus E$.
For each choice of $\{V_n\}_n$ we define
$$
\aligned D_{S^*}=D_{S^*}(\{V_n\}_n):=\sup_{n,i,j}\big\{\,
 & d_{S^*}|_{V_n\setminus E_n}(\eta^n_i,\eta^n_j)\; : \;
 \eta^n_i,\eta^n_j
\text{ are different connected components of $\p V_n$} \\
& \; \text{and } \; \eta^n_i,\eta^n_j \text{ are in the same
connected component of } S\setminus\inte V_n\,
 \big\}.
\endaligned $$
\end{definicion}

\mpb

\begin{obs}
(1) Note that if $\eta^n_i,\eta^n_j$ are in the same connected
component of $S\setminus\inte V_n$, then $S \setminus \eta_i^n$ is
connected.

(2) Recall that $d_{S^*}\neq d_S|_{S^*}$, since $(S^*,d_{S^*})$ is a
geodesically complete Riemannian manifold (the points of $E$ are at
infinite $d_{S^*}$-distance of the points of $S^*$; in fact, $S^*$
is an open non-exceptional Riemann surface).
\end{obs}

\mpb

The following results show the relevance of $D_{S^*}(\{V_n\}_n)$
(see also Theorem \ref{t:main}).

\begin{prop}
\emph{(\cite[Proposition 5.1]{PRT3})} \label{p:prt}
 Let $S$ be an open
non-exceptional Riemann surface, $E=\cup_nE_n$ an
$(r,s,N)$-uniformly separated set in $S$ and $S^*:=S \setminus E$.
Let us assume also that we can choose the sets $\{V_n\}_n$ such
that $D_{S^*}(\{V_n\}_n)=\infty$. Then $S^*$ is not hyperbolic.
\end{prop}

\begin{teo}
\emph{(\cite[Theorem 5.4]{PRT3})} \label{t:rsn} Let $S$ be an open
non-exceptional Riemann surface and $E=\cup_nE_n$ an
$(r,s,N)$-uniformly separated set in $S$. Then, $S^*:=S\setminus E$
is $\d^*$-hyperbolic if and only if $S$ is $\d$-hyperbolic,
$D_{S^*}(\{V_n\}_n)$ is finite and $V_n\setminus E_n$ is
$k$-hyperbolic for every $n$ (with $d_{S^*}|_{V_n\setminus E_n}$).

Furthermore, if $D_{S^*}(\{V_n\}_n)$ is finite and $V_n\setminus
E_n$ is $k$-hyperbolic for every $n$, then $\d^*$ (respectively
$\d$) is a universal constant which only depends on
$r,s,N,k,D_{S^*}(\{V_n\}_n)$ and $\d$ (respectively
$r,s,N,D_{S^*}(\{V_n\}_n)$ and $\d^*$).
\end{teo}

In the above theorem $\{ V_n\setminus E_n\}_n$ is a family of
bordered non-exceptional Riemann surfaces, all isometrically
embedded into $S^*$, and this family is required to be uniformly hyperbolic.
This uniform hyperbolicity condition will be removed in
Section~\ref{section5}.

\spb

If $S$ has no genus, then the set in which we take the supremum that
defines $D_{S^*}$ is the empty set. Hence, we deduce the following
direct consequence.

\begin{coro}
\label{c:rsn} Let $S$ be an open non-exceptional Riemann surface
with no genus, and $E=\cup_nE_n$ an $(r,s,N)$-uniformly separated
set in $S$. Then, $S^*:=S\setminus E$ is $\d^*$-hyperbolic if and
only if $S$ is $\d$-hyperbolic and $V_n\setminus E_n$ is
$k$-hyperbolic for every $n$ (with $d_{S^*}|_{V_n\setminus E_n}$).

Furthermore, if $V_n\setminus E_n$ is $k$-hyperbolic for every $n$,
then $\d^*$ (respectively $\d$) is a universal constant which only
depends on $r,s,N,k$ and $\d$ (respectively $r,s,N$ and $\d^*$).
\end{coro}

Finally we include a technical result about the Poincar\'e metric.

\begin{lema}
\label{l:cociente} \emph{(\cite[Lemma 3.1]{APR})} Let us consider an
open non-exceptional Riemann surface $S$, a closed non-empty subset
$C$ of $S$, and a positive number $\e$. If $S^*:=S\setminus C$, then
we have that $1 < L_{S^*}(\g)/L_S(\g) < \coth (\e/2)$, for every
curve $\g\subset S$ with finite length in $S$ such that
$d_S(\g,C)\ge\e$.
\end{lema}

\spb

\section{Stability of hyperbolicity.}\label{section5}

The leading idea in this Section is that some quantitative
information that seems to influence hyperbolicity of a surface
actually is irrelevant, let us see an example. If $S$ is an open
Riemann surface and $p_1,p_2\in S$, then several conformal
invariants of $S^*=S\setminus\{ p_1,p_2\}$ (e.g. the exponent of
convergence, the first eigenvalue of the Laplace-Beltrami operator
and the isoperimetric constant) degenerate when $p_2$ tends to
$p_1$; in contrast, the hyperbolicity constant stays bounded
(stable) as $p_1$ approaches~$p_2$.

\spb

In this Section we only consider surfaces without genus, so that in
particular the number $D_{S^*}$ from Definition~\ref{def:D} is zero.
We begin by proving Theorem~\ref{t:finite} as a surprising
consequence of Theorem~\ref{t:balls} (on the topology of balls).
Then Corollary~\ref{c:finite}, an immediate consequence of
Theorem~\ref{t:finite}, is used to prove Theorem~\ref{t:clasef}.

In its turn, Theorem~\ref{t:clasef} is fundamental for the proof of
the main Theorem in Section~\ref{section:6}.

\begin{teo}
\label{t:finite} Let us consider a $\d$-hyperbolic non-exceptional
Riemann surface $S$ with no genus, and pairwise disjoint simply
connected compact sets $\{E_n\}_{n=1}^N$ in $S$. We define
$S^*:=S\setminus \cup_{n=1}^N E_n$. Assume that for each
$n=1,\dots,N,$ there exists a simple closed curve $g_n$ `surrounding
just $E_n$' with $L_{S^*}(g_n) \le l$. Then there exists a constant
$\d^*$, which only depends on $\d$, $N$ and $l$, such that $S^*$ is
$\d^*$-hyperbolic.
\end{teo}

\begin{obs}
By $g_n$ `surrounding just $E_n$' we mean that $g_n$ is
homotopically trivial in $S$, $g_n$ surrounds $E_n$ and $g_n$ does
not surround $E_k$ for $k\neq n$ ($g_n$ is `freely homotopic' to
$E_n$). Note also that $g_n \cap (\cup_k E_k)=\varnothing$, since
$L_{S^*}(g_n) < \infty$.
\end{obs}

\begin{proof}
First we prove the case $N=1$.

We have $L_{S}(g_1) < L_{S^*}(g_1) \le l$.
The curve $g_1$ surrounds a simply connected open set
$D\subset S$ with $E_1 \subset D$; then we can lift $D$ to $\tilde
D \subset \DD$ and given $z,w\in \tilde{D}$, we consider the
infinite geodesic $\eta$ in $\DD$ joining $z,w$; the geodesic
$\eta$ meets $\p \tilde D$ in $z',w'$ with $[z,w]\subset [z',w']$;
therefore, $d_{\DD}(z,w) \le
d_{\DD}(z',w') \le L_{\DD}(\tilde{g_1})/2 = L_{S}(g_1)/2 \le l/2$,
and $\diam_{S}(D) \le \diam_{\DD}(\tilde{D}) \le l/2$.
Hence, $\diam_{S}(E_1) \le l/2$.

Let us fix any $p\in E_1$; then $E_1 \subset D \subset \overline{B_S(p,l/2)}$.
Since $K\equiv -1$, by Theorem \ref{t:balls} we know that there
exists $l'\in [l,l+1]$ such that $\p B_S(p,l')$ is a union of simple
closed curves and
$$
\hbox{rank}\,H_1 \big(B_S(p,l')\big) \le \frac{\sinh(l+1)}{\sinh 1}
< e^{l}/(1-e^{-2}) < 2 \, e^l .
$$
We define $V_1$ as the closure of the ball $B_S(p,l')$; then $g_1$
is contained in $V_1$. Consequently, $d_S(E_1, \p V_1) \ge l/2$. We
also have $L_S(\p V_1) \le L_\DD(B_\DD(0,l')) = 2 \pi \sinh l'\le 2
\pi \sinh (l+1)$. The boundary $\partial V_1$ has at most $1+2e^l$
connected components, moreover $V_1$ has no genus because the
ambient surface $S$ has zero genus by hypothesis. All this implies
that $V_1$ is an $(1+2e^l)$-normal neighborhood of $E_1$ in $S$, and
$E_1$ is therefore an $(l/2, 2 \pi \sinh (l+1), 1+2e^l)$-uniformly
separated set in $S$.

We check now that $V_1^*:=V_1 \setminus E_1$ is hyperbolic with the
intrinsic distance $d_{S^*}|_{V_1^*}$ induced on $V_1^*$ by the Poincar\'e
metric of $S^*=S \setminus E_1$. By Lemma \ref{l:cociente}, since
$d_S(E_1, \p V_1) \ge l/2$,
$$
L_{S^*}(\p V_1) \le L_{S}(\p V_1) \coth (l/4) \le 2 \pi \sinh (l+1)
\coth (l/4) \,.
$$
Each component of $\partial V_1 =\partial V_1^*$ gives rise to an
outer loop in $V_1^*$. If $E_1$ is a single point then $V_1^*$ has a
puncture at $E_1$; otherwise $V_1^*$ has one additional outer loop
freely homotopic to $g_1$. In any case the sum of the number of
outer loops plus the number of punctures in $V_1^*$ is at most
$2+2e^l$. The $S^*$-length of the outer loops coming from $\partial
V_1$ is less than or equal to $2 \pi \sinh (l+1) \coth (l/4)$.

Since $L_{S^*}(g_1) \le l$, if $E_1$ is not a puncture then the
outer loop in $V_1^*$ homotopic to $g_1$ has length less than or
equal to $l<2 \pi \sinh (l+1) \coth (l/4)$
(since $g_1$ is contained in $V_1^*$).
Consequently $V_1^* \in
\S(2e^l ,2 \pi \sinh (l+1) \coth (l/4))$, and Theorem \ref{t:clases}
says that there exists a constant $\d_1$, which only depends on $l$,
such that $V_1^*$ is $\d_1$-hyperbolic.

The hypothesis of $S$ having genus zero allows us to use Corollary
\ref{c:rsn}, hence there exists a constant $\d_1^*$, which only
depends on $\d_1$ and $l$, such that $S^*$ is $\d_1^*$-hyperbolic.
This finishes the proof in the case $N=1$.

\spb

Now we prove the result by induction on $N$. We have proved it for
$N=1$. Assume that it holds for $N-1$ (note that we also have
$L_{S\setminus (E_1\cup \cdots\cup E_{N-1})}(g_n) < L_{S\setminus
(E_1\cup \cdots\cup E_{N})}(g_n) \le l$ for $n=1,\dots,N-1$).
Consequently, $S\setminus (E_1\cup \cdots\cup E_{N-1})$ is
$\d_{N-1}$-hyperbolic, where $\d_{N-1}$ only depends on $\d$, $N$
and $l$.

Since $g_N$ is a simple closed curve surrounding just $E_N$, with
$L_{S^*}(g_N) \le l$, the result for $N=1$ gives that $S^*$ is
$\d^*$-hyperbolic, with $\d^*$ a constant which only depends on
$\d$, $N$ and $l$.
\end{proof}

In Theorem~\ref{t:infinite} we shall extend Theorem~\ref{t:finite}
to the case of infinitely many sets $E_n$. It needs an extra
hypothesis: that the $E_n$'s get neither too small nor too large as
$n\to\infty$.

\vspace{3mm}

Since any puncture can be surrounded by arbitrarily short closed
curves, we deduce the following result:

\begin{coro}
\label{c:finite0} Let us consider a $\d$-hyperbolic non-exceptional
Riemann surface $S$ with no genus, and points $\{p_n\}_{n=1}^N$ in
$S$. We define $S^*:=S\setminus \{p_1,\dots,p_N\}$. Then there
exists a constant $\d^*$, which only depends on $\d$ and $N$, such
that $S^*$ is $\d^*$-hyperbolic.
\end{coro}

Corollary \ref{c:finite0} can be viewed as a result on stability
of hyperbolicity: $S^*$ is $\d^*$-hyperbolic independently of how
close or far apart the points $\{p_1,\dots, p_N\}$ are from one
another.

\spb

Since $\DD$ is hyperbolic, Theorem~\ref{t:finite} also implies the
following.

\begin{coro}
\label{c:finite} Let us consider pairwise disjoint simply connected
compact sets $\{E_n\}_{n=1}^N$ in $\DD$ and $\DD^*:=\DD \setminus
\cup_{n=1}^N E_n$. Assume that for each $n=1,\dots,N,$ there exists
a simple closed curve $g_n$ surrounding just $E_n$ with
$L_{\DD^*}(g_n) \le l$. Then there exists a constant $\d^*$, which
only depends on $N$ and $l$, such that $\DD^*$ is $\d^*$-hyperbolic.
\end{coro}

Finally we prove the following improvement of Theorem
\ref{t:clases}. It is surprising since we do not require anything
about one of the outer loops.

\begin{teo}
\label{t:clasef} For each $a$ and $l$, there exists a constant $\d$,
which just depends on $a$ and $l$, such that every $S\in \F(a,l)$ is
$\d$-hyperbolic.
\end{teo}

\begin{obs}
It is interesting to note that it is not possible to obtain a similar result to
Theorem \ref{t:clasef} if all the outer loops except two have
bounded length, as the following example shows:
if $Y_t$ is the $Y$-piece with simple closed geodesics
$\g_1\cup \g_2\cup \g_3 = \p Y_t$
such that $L(\g_1)=1$ and $L(\g_2)=L(\g_3)=t$,
then $\lim_{t\to\infty} \d(Y_t)=\infty$.
\end{obs}

\begin{proof}
We first prove the result for open surfaces.

If $S\in \S(a,l)$, then Theorem \ref{t:clases} gives the result;
this happens in particular when $a=0$. Therefore, we can assume that
$\chi (S)<0$, that an outer loop $\g_0$ satisfies $L_S(\g_0)>l$, and
that any other outer loop $\g_j$ $(j=1,\dots,N)$ verifies
$L_S(\g_j)\le l$. From $-a\leq \chi(S)=2-(N+1)<0$ we infer $2\leq
N\leq a+1$.

For open surfaces the conformal structure and the Riemann metric
determine each other, so we can consider one structure or the other
to our convenience. Having zero genus, $S$ can be represented as a
plane domain $S\subset \CC$ with $S=\O \setminus E_1\cup \cdots \cup
E_N$, $\O$ a simply connected open set, $E_1, \dots, E_N$ simply
connected compact sets, such that $\g_0$ surrounds $E_1\cup \cdots
\cup E_N$ and $\g_j$ surrounds just $E_j$ ($j=1, \dots ,N$). The
hypothesis $L_S(\g_0)>l$ implies that $\g_0$ is not a puncture and
that $\O\neq \CC$; then, by the Riemann mapping Theorem, we can
assume that $S=\DD \setminus E_1\cup \cdots \cup E_N$. Since we have
$N \le a+1$ and $L_S(\g_j)\le l$ ($j=1, \cdots , N$), by Corollary
\ref{c:finite} there exists a constant $\d$, which just depends on
$a$ and $l$, such that $S$ is $\d$-hyperbolic.

\spb

We now prove the result for bordered surfaces.

The idea of the proof is to see a bordered surface in $\F(a,l)$ as a
subset of an open surface in $\F(a,l)$, and then make use of Theorem
\ref{t:treedecomp}.

Given the bordered surface $S$, consider the open surface $R_S$ from
Proposition~\ref{RS}. If $R_S$ is the unit disk then it is
$\log(1+\sqrt2\,)$-thin (see e.g. \cite[p.130]{An}).

Assume now that $R_S$ is not the unit disk. Outer loops in $S$ are
metric geodesics (recall Definition~\ref{def:Rips}), perhaps not
Riemannian geodesics, but they give rise in $R_S$ to Riemannian
outer loops of no greater length (including punctures), or they just
shrink to points in $R_S$. Hence $R_S\in\F (a,l)$, and by the open
case there is a constant $\d_1$, just depending on $a$ and $l$, such
that $R_S$ is $\d_1$-hyperbolic.

The closure of $R_S\setminus S$ is the union of simply or doubly
connected bordered surfaces $R^1,\dots,R^s,$ with $s\le a+2$, and
the condition $L_S(\p S)\le l$ implies that $\{S,R^1,\dots,R^s\}$ is
an $l$-tree decomposition of $R$. Then, by Theorem
\ref{t:treedecomp}, there exists a constant $\d$ that depends only
on $a$ and $l$ and such that $S$ is $\d$-hyperbolic.
\end{proof}

\section{Main results on hyperbolicity}\label{section:6}

Now, taking advantage of all the tools developed in the previous
Sections, we present the main results on hyperbolicity of the paper.
The first Theorem we present improves Theorem \ref{t:rsn} by
removing the uniform hyperbolicity hypothesis, which is usually the
hardest one to check.

\begin{teo}
\label{t:main} Let $S$ be an open non-exceptional Riemann surface
and $E=\cup_n E_n$ a $(r,s,N)$-uniformly separated set in $S$, with
$E_n$ simply connected for every $n$. Then, $S^*:=S \setminus E$ is
$\d^*$-hyperbolic if and only if $S$ is $\d$-hyperbolic and the
number $D_{S^*}(\{V_n\}_n)$ from Definition~\ref{def:D} is finite.

Furthermore, $\d^*$ (respectively $\d$) is a universal constant
which only depends on $r,s,N,D_{S^*}(\{V_n\}_n)$ and $\d$
(respectively $\d^*$).
\end{teo}

\begin{obs}

\begin{enumerate}
\item Recall that if $E_n$ is simply connected, then it gives rise
to either a puncture (if $E_n$ is a single point) or a funnel (if
$E_n$ is not a single point) in $S^*$. \item Note that we do not
require anything about $\diam_S E_n$; in particular, we allow the
case $\sup_n \diam_S E_n= \infty$; in this case  the funnels $F_n$
in $S^*$ corresponding to $E_n$ verify $\sup_n L_{S^*}(\p F_n) \ge
\sup_n L_{S}(\p F_n) \ge \sup_n \diam_S E_n = \infty$, which makes
the study of the hyperbolicity of $S^*$ more difficult. \item
Theorem \ref{t:main} is a known result in the particular case when
every $E_n$ is a single point (see \cite[Theorem 3.1]{PRT2}).
\end{enumerate}
\end{obs}

\begin{proof}
In order to apply Theorem \ref{t:rsn}, we just need to prove that
$V_n^*:=V_n\setminus E_n$ is $k$-hyperbolic for every $n$, where $k$
is a constant which only depends on $r,s$ and $N$.

Recall that $V_n$ is compact and belongs to $\S(N-2,s)\subset
\F(N-2,s)$ for any $n$.

If $\p V_n$ is a single closed curve (i.e. $V_n$ is a topological
disk), then there is just one outer loop in $V_n^*$ and, by Lemma
\ref{l:cociente}, $L_{S^*}(\p V_n) < L_{S}(\p V_n) \coth (r/2)\le s
\coth (r/2)$. Hence, $V_n^*\in\S(0,s \coth (r/2))=\F(0,s \coth
(r/2))$. In this case Theorem \ref{t:clases} suffices to ensure that
$V_n^*$ is $k_1$-hyperbolic, with a constant $k_1$ which only
depends on $r$ and $s$.

If $\p V_n$ is not connected, let us denote by $\g_n$ the simple
closed geodesic in $V_n^*$ which surrounds just $E_n$ (if $E_n$ is a
single point, as usual, we see $\g_n$ as a puncture and
$L_{S^*}(\g_n)=0$).

Note that any outer loop $\g$ distinct from $\g_n$ in $V_n^*$ is
freely homotopic to some closed curve in $\p V_n$.

Since Lemma \ref{l:cociente} implies $L_{S^*}(\p V_n)< L_{S}(\p V_n)
\coth (r/2)\le s \coth (r/2)$, we deduce that $V_n^*\in\F(N-1,s
\coth (r/2))$ (recall that $V_n$ has at most $N$ outer loops and
$E_n$ is simply connected for every $n$; we do not need to bound the
length of the outer loop corresponding to $E_n$). Theorem
\ref{t:clasef} guarantees that $V_n^*$ is $k_2$-hyperbolic, with a
constant $k_2$ which only depends on $r,s$ and $N$.

Now Theorem \ref{t:rsn} gives the result.
\end{proof}

We would like not to have to check the hypothesis
$D_{S^*}(\{V_n\}_n)< \infty$. The two following results allow to
remove this hypothesis if $S$ has either no genus or finite genus.

\medskip

If $S$ has no genus, then the set in which we take the supremum in
order to define $D_{S^*}$ is the empty set. Hence, we deduce the
following direct consequence.

\begin{coro}
\label{c:main} Let $S$ be an open non-exceptional Riemann surface
with no genus and $E=\cup_nE_n$ an $(r,s,N)$-uniformly separated
set in $S$, with $E_n$ simply connected for every $n$. Then,
$S^*:=S \setminus E$ is $\d^*$-hyperbolic if and only if $S$ is
$\d$-hyperbolic.

Furthermore, $\d^*$ (respectively $\d$) is a universal constant
which only depends on $r,s,N$ and $\d$ (respectively~$\d^*$).
\end{coro}

\begin{teo}
\label{t:finitegenus} Let $S$ be an open non-exceptional Riemann
surface with finite genus and $E=\cup_nE_n$ an $(r,s,N)$-uniformly
separated set in $S$, with $E_n$ simply connected for every $n$.
Then, $S^*:=S \setminus E$ is hyperbolic if and only if $S$ is
hyperbolic.
\end{teo}

\begin{proof}
If $S$ has no genus, then Corollary \ref{c:main} gives the result.
Therefore, we can assume that $S$ has genus.

Given a choice of $\{V_n\}_n$, we define the subset of indices
$\Lambda$ as the set of all $n$ such that there are different
connected components $\eta^n_i,\eta^n_j$ of $\p V_n$ in the same
connected component of $S\setminus\inte V_n$. We are going to prove
that $\Lambda$ is finite. This will imply that $D_{S^*}(\{V_n\}_n)$
is the maximum of at most $\frac{N(N-1)}{2}\cdot\card\Lambda\,$
finite distances, hence finite, and then Theorem \ref{t:main} will
finish the proof.

Since $S$ has finite genus, there exists a domain $G\subset S$
verifying the following facts: $\overline{G}$ is a compact set whose
boundary is a finite collection $g_1,\dots ,g_h$ of simple closed
curves, and $S\setminus G$ is a disjoint union $S_1\cup\cdots\cup
S_h$ where each $S_j$ is a bordered surface with no genus, and $\p
S_j=S_j\cap\overline{G}=g_j$ for $j=1,\dots ,h$.

Only finitely many of the $V_n$ intersect $\overline{G}$; otherwise
the condition $d(V_n,V_m)\ge r$ could not be satisfied for all
$n\neq m$. We next prove that if $V_n\cap\overline{G}=\varnothing$
then $n\notin\Lambda$, and finiteness of $\Lambda$ follows.

Let us suppose $V_n$ is disjoint from $\overline{G}$, that two
connected components $\eta_{ni},\eta_{ni'}$ of $\p V_n$ can be
connected in $S\setminus \inte V_n$, and derive a contradiction.
Since $V_n$ is connected, it is contained in one of the $S_j$, say
$S_{j_0}$. The path $\gamma$ connecting $\eta_{ni}$ to $\eta_{ni'}$
in $S\setminus \inte V_n$ cannot be all inside $S_{j_0}$, for that
would force $S_{j_0}$ to have genus. Therefore $\gamma$ exits
$S_{j_0}$, which it can do only by crossing $g_{j_0}$. Any time
$\gamma$ re-enters $S_{j_0}$, it must do so by crossing $g_{j_0}$.
One concludes that the parts of $\gamma$ outside $S_{j_0}$ can be
replaced with arcs of $g_{j_0}$, but this yields a continuous path
in $S_{j_0}\setminus\inte V_n$ which joins $\eta_{ni}$ to
$\eta_{ni'}$, once again forcing $S_{j_0}$ to have genus.

We conclude that $\Lambda\subseteq\{\, n\; :\; V_n\cap
G\neq\varnothing\,\}$; hence $\Lambda$ and $D_{S^*}(\{V_n\}_n)$ are
finite, as was to be proved
\end{proof}

Finally we shall prove Theorem~\ref{t:infinite}, a complementary
result to Corollary \ref{c:main} with very different hypotheses. It
is also the $N=\infty$ analogue of Theorem \ref{t:finite}. Removing
an infinity of sets $E_n$ from the initial surface $S$ can ruin
hyperbolicity if the $E_n$ become too small or too large as
$n\to\infty$. One idea is to reduce to the case $S={\mathbb D}$ and
then use annuli, instead of curves, to `surround' the sets $E_n$.
More concretely, the condition that the domain ${\mathbb D}\setminus
E$ must satisfy is having \emph{uniformly perfect boundary}, which
we define below. We also quote some results, about the Poincar\'e
and quasihyperbolic metrics of a domain, that are used in the final
proof.

\begin{definicion}
A \emph{generalized annulus} $\O$ is a doubly connected open subset
of the complex plane which is not the plane minus a point; then its
complement (in the Riemann sphere) has two connected components.

Given any generalized annulus $\O$, there exists a conformal mapping
of $\O$ onto $\{z\in\CC: \, 1< |z-a|<R \}$, for some $1<R \le
\infty$. We define the \emph{modulus} of $\O$ as
$$
\modulus \O := \frac1{2\pi} \log R \,.
$$
We say that a generalized annulus $\O$ \emph{separates} a closed set
$E$ if $\O$ does not intersect $E$ and each connected component of
the complement of $\O$ intersects $E$.

We say that $E$ is \emph{uniformly perfect} if there exists a
constant $c_1$ such that $\modulus \O \le c_1$ for every generalized
annulus separating $E$ (see \cite{BP}).
\end{definicion}

Two useful properties of the modulus are the following:

(A) If $\g$ is the simple closed geodesic for the Poincar\'e metric
in $\O$, then $\modulus \O =\pi/L_\O(\g)$ (if $\O$ has a puncture we
can see $\g$ as the puncture and then $L_{\O}(\g)=0$ and $\modulus
\O = \infty$).

(B) If $\O_1 \subseteq \O_2$, then $\modulus \O_1 \le \modulus
\O_2$.

\spb

A domain with one or more punctures is never uniformly perfect. If
we remove from the unit disk $\mathbb D$ a sequence of straight
segments whose lengths converge to zero, then the resulting domain
is not uniformly perfect. This example leads to the hypothesis
$\diam_SE_n\ge c$ in the statement of Theorem~\ref{t:infinite}.

\spb

Uniformly perfect sets verify the following interesting property:

\begin{teo}
\emph{(\cite[Corollary 1]{BP})} \label{BPthm2} Let $\O\subset\CC$ be
a non-exceptional domain and $ds=\lambda_\O (z)\, |dz|$ its
Poincar\'e metric. Define also $\d_\O (z):=\min\{|z-a|: \, a \in \p
\O \}$. The following conditions are equivalent:
\begin{enumerate}
\item There exists a positive constant $c_2$ with $\,\displaystyle
\frac{c_2}{\d_\O (z)} \le \l_\O (z) \le \frac2{\d_\O (z)}\, $ for
every $z\in\O$ \item $\p\O$ is uniformly perfect.
\end{enumerate}

\noindent Furthermore, if $\p\O$ is uniformly perfect then the
constant $c_2$ just depends on the uniformly perfect constant of
$\p\O$.
\end{teo}

If we define, as usual, the \emph{quasihyperbolic length} of a curve
$\gamma$ as
$$
k_\O (\g) := \int_\g \frac{|dz|}{\d_\O (z)} \;,
$$
then Theorem \ref{BPthm2} says that $\p\O$ is uniformly perfect if
and only if $L_\O (\g) \ge c_2 k_\O (\g)$ for every curve $\g
\subset \O$.

\spb

We need one more technical result:

\begin{lema}
\label{minLenLem} \emph{(\cite[Lemma 3.3]{HLPRT})} Let $\gamma$ be a
curve in a domain $\O\subset \RR^n$ starting at a point $x$ and with
Euclidean length~$s$. Then
$$
k_\O(\gamma) \ge \log \Big( 1 + \frac{s}{\d_\O(x)} \Big)\,.
$$
\end{lema}

\begin{teo}
\label{t:infinite} Let us consider an open non-exceptional Riemann
surface $S$ with no genus, and pairwise disjoint simply connected
compact sets $\{E_n\}_{n=1}^\infty$ in $S$ with $\diam_S E_n \ge c$
for every $n$. We define $S^*:=S\setminus \cup_{n=1}^\infty E_n$.
Assume that for each $n$ there exists a simple closed curve $g_n$
surrounding just $E_n$ with $L_{S^*}(g_n) \le l$. Then $S$ is
$\d$-hyperbolic if and only if $S^*$ is $\d^*$-hyperbolic.

Furthermore, $\d^*$ (respectively $\d$) is a universal constant
which only depends on $c$, $l$ and $\d$ (respectively $\d^*$).
\end{teo}

\begin{obs}
The conclusion of Theorem \ref{t:infinite} does not hold if we
remove either the hypothesis $\diam_S E_n \ge c$ or $L_{S^*}(g_n)
\le l$.
\end{obs}

\begin{proof}
For each $n$ there exists a simple closed geodesic $\g_n$
surrounding just $E_n$ (freely homotopic to $g_n$) with
$L_{S^*}(\g_n) \le L_{S^*}(g_n) \le l$ ($\g_n$ can not be a puncture
since $\diam_S E_n \ge c$ implies that $E_n$ is not an isolated
point).

\vspace{3mm}

\noindent{\bf Claim:}  there exists a positive constant $\e_0$,
which just depends on $c$ and $l$, such that $d_S(E_n,\g_n) \ge
\e_0$ for every $n$.

\vspace{3mm}

We prove the Theorem assuming this claim. Let $V_n$ be the closure
of the simply connected open subset of $S$ surrounded by $\g_n$.
Then $V_n$ is a $1$-normal neighborhood of $E_n$. Furthermore,
$$
\begin{aligned}
d_S(E_n,\p V_n) & = d_S(E_n,\g_n) \ge \e_0 \,,
\\
L_S(\p V_n)=L_S(\g_n) & < L_{S^*}(\g_n) \le L_{S^*}(g_n) \le l \,.
\end{aligned}
$$
Given $n\neq m$, we have $L_{S^*}(\g_n), L_{S^*}(\g_m) \le l$, and
by the Collar Lemma (see \cite{R}) there are collars in $S^*$ around
$\g_n$ and $\g_m$ of width $\Arccosh \coth (l/2)$. These collars are
pairwise disjoint, as explained in Remark~\ref{disjuntas}. Then we
deduce that $d_{S^*}(\g_n,\g_m) \ge 2 \Arccosh \coth (l/2)$. Now
Lemma \ref{l:cociente} gives:
$$
d_{S}(V_n,V_m) = d_{S}(\g_n,\g_m) > \tanh(\e_0/2) d_{S^*}(\g_n,\g_m)
\ge 2 \tanh(\e_0/2) \Arccosh \coth (l/2)\,.
$$
If we define $r := \min\{\e_0, 2 \tanh(\e_0/2) \Arccosh \coth
(l/2)\}$, then $\{E_n\}_n$ is a $(r,l,1)$-uniformly separated set
in~$S$.

Corollary \ref{c:main} states that $S$ is $\d$-hyperbolic if and
only if $S^*$ is $\d^*$-hyperbolic, with the appropriate behaviour
of the constants. This finishes the proof if the claim holds.

\vspace{3mm}

\noindent Now we are going to prove the claim.

We prove first that without loss of generality we can assume
$S=\DD$: Let $\pi:\DD\longrightarrow S$ be a universal covering map.
We consider $F:=\pi^{-1}(E)$ and the connected components $\{F_n\}$
of $F$. If $\DD^*:=\DD\setminus F$, then $\pi:\DD^*\longrightarrow
S^*$ is also a covering map. Consequently, $\pi$ defines two local
isometries: $\DD\to S$ and $\DD^*\to S^*$. Given any fixed $F_n$
then $E_m:=\pi(F_n)$ verifies that $\pi:F_n\to E_m$ is a bijection.
Hence, $\diam_\DD(F_n) \ge \diam_S(E_m) \ge c$.

Let $W_n$ be the connected component of $\pi^{-1}(V_m)$ containing
$F_n$. Then $\g'_n=\partial W_n$ is the simple closed geodesic in $\DD^*$
surrounding just $F_n$ and $L_{\DD^*}(\g'_n) = L_{S^*}(\g_m) \le l$.
Consequently, $\{F_n\}_{n=1}^\infty$ verifies the hypotheses in
Theorem \ref{t:infinite}. Since $\pi$ defines bijections $W_n\to
V_m$ and $\g'_n\to\g_m$, we have $d_\DD(F_n,\g_n') =d_S(E_m,\g_m)$.
In order to prove the claim we can thus assume without loss of
generality that $S=\DD$, and $\DD^*=\DD\setminus \cup_{n=1}^\infty
E_n$.

\spb

We prove now that $\p \DD^*$ is a uniformly perfect set: Let us
consider a generalized annulus $A$ separating $\p \DD^*$; then
$A\subset \DD^*\subset \DD$ and the bounded connected component of
the complement of $A$ contains some $E_{n_0}$. Hence, $A\subset \DD
\setminus E_{n_0}$ and consequently $\modulus A\le \modulus (\DD
\setminus E_{n_0})$.

We are looking for a lower bound of $L_{\DD \setminus
E_{n_0}}(\eta_{n_0})$, where $\eta_{n_0}$ is the simple closed
geodesic for the Poincar\'e metric in $\DD \setminus E_{n_0}$.
Consider $a,b \in E_{n_0}$ with $d_\DD(a,b)=\diam_\DD E_{n_0}$, and
the simple closed geodesic $\eta$ for the Poincar\'e metric in $\DD
\setminus \{a,b\}$. We have that
$$
L_{\DD \setminus E_{n_0}}(\eta_{n_0})
> L_{\DD \setminus \{a,b\}}(\eta_{n_0})
\ge L_{\DD \setminus \{a,b\}}(\eta) \,.
$$
Since $\DD \setminus \{a,b\}$ and $\DD \setminus \{a',b'\}$ are
isometric if and only if $d_\DD(a,b)=d_\DD(a',b')$, then $L_{\DD
\setminus \{a,b\}}(\eta)=f(d_\DD(a,b))$, for some function $f:
(0,\infty) \longrightarrow (0,\infty)$. Since $\eta$ surrounds
$\{a,b\}$,
$$
f(d_\DD(a,b)) = L_{\DD \setminus \{a,b\}}(\eta)
> L_{\DD}(\eta) > d_\DD(a,b) \,;
$$
consequently, $f(t)>t$ and since $d_{\DD} (a,b)=\diam_\DD E_n \ge
c$,
$$
L_{\DD \setminus E_{n_0}}(\eta_{n_0}) > L_{\DD \setminus
\{a,b\}}(\eta) = f(d_{\DD} (a,b)) > d_{\DD} (a,b) \ge c \,,
$$
and hence $\modulus A \le \modulus (\DD \setminus E_{n_0}) \le
\pi/c$. This shows that $\p \DD^*$ is a uniformly perfect set.

\medskip

Now, by Theorem \ref{BPthm2}, there exists a constant $c_1$, which
just depends on $c$, such that $L_{\DD^*} (\g) \ge c_1 k_{\DD^*}
(\g)$ for every curve $\g \subset {\DD^*}$.

Let us consider a fixed $n$ and the simple closed geodesic $\g_n$ in
$\DD^*$ surrounding just $E_n$ (freely homotopic to $g_n$) with
$L_{S^*}(\g_n) \le L_{S^*}(g_n) \le l$. Take $p \in E_{n}$ and $q
\in \g_{n}$ with $\e:=d_\DD(E_n,\g_n)=d_\DD(p,q)$. Since
$d_\DD(0,a)=2\Arctanh a$, using a M\"{o}bius map if it is necessary,
we can assume without loss of generality that $p=0$ and $q=\tanh
(\e/2)$. Since $L_{\DD}(\g_n)> \diam_\DD E_{n} \ge c$, we can
consider a subcurve $\g_n^0 \subset \g_n$ starting at $q$ and with
$L_{\DD}(\g_n^0) = c$; then,
$$
\g_n^0 \subset \overline{B_\DD(0,\e+c)} = \overline{B_{Eucl}\big(\,
0\, ,\,\tanh ((\e+c)/2)\,\big)} \,.
$$
Then
$$
L_{Eucl}(\g_n^0) = \int_{\g_n^0} |dz| \ge \frac{1-\tanh\!^2
((\e+c)/2)}2 \int_{\g_n^0} \frac{2\,|dz|}{1-|z|^2} =
\frac{L_{\DD}(\g_n^0)}{2 \cosh\!^2((\e+c)/2)} = \frac{c}{2 \cosh\!^2
((\e+c)/2)} \;.
$$
Therefore, applying Lemma \ref{minLenLem}, we obtain
$$
\begin{aligned}
l & \ge L_{\DD^*}(g_n) \ge L_{\DD^*}(\g_n) \ge L_{\DD^*}(\g_n^0) \ge
c_1 k_{\DD^*}(\g_n^0)
\\
& \ge c_1 \log \Big( 1 + \frac{L_{Eucl}(\g_n^0)}{\d_{\DD^*}(q)}
\Big) \ge c_1 \log \Big( 1 + \frac{L_{Eucl}(\g_n^0)}{q} \Big)
\\
& \ge c_1 \log \Big( 1 + \frac{c}{2 \tanh (\e/2) \cosh\!^2
((\e+c)/2)} \Big)\,.
\end{aligned}
$$
Hence,
$$
2 \tanh (\e/2) \cosh\!^2 ((\e+c)/2) \ge \frac{c}{e^{l/c_1}-1} \;.
$$
Note that, for each fixed $c$, the function $f_c: (0,\infty)
\longrightarrow (0,\infty)$ given by $f_c(\e)=2 \tanh (\e/2)
\cosh\!^2 ((\e+c)/2)$ is positive and increasing in $\e \in
(0,\infty)$. If we define
$$
\e_0:= f_c^{-1} \Big( \frac{c}{e^{l/c_1}-1} \Big) >0 \,,
$$
then $\e_0$ just depends on $c$ and $l$, and $d_\DD(E_n,\g_n)=\e \ge
\e_0$ for every $n$. This finishes the proof of both the claim and
the Theorem.
\end{proof}

\

\

\begin{tabular}{cccc}
\small
\parbox{7cm}{
Ana Portilla, Jos\'e M. Rodr\'{\i}guez, Eva Tour\'{\i}s\\
Departamento de Matem\'aticas\\
Escuela Polit\'ecnica Superior\\
Universidad Carlos III de Madrid \\
Avenida de la Universidad, 30\\
28911 Legan\'es (Madrid), SPAIN\\} & &\hspace{2cm} & \small
\parbox{6cm}{
\quad\\
Jes\'us Gonzalo\\
Departamento de Matem\'aticas\\
Facultad de Ciencias\\
Universidad Aut\'onoma de Madrid\\
28049 Madrid, SPAIN\\
$\qquad$\\}
\end{tabular}

\enddocument